\documentclass{article}
\def\nkyear{2026}            
\def\nkvol{xx}                
\def\nkno{yy}                 
\def\firstpage{1}            
\def\lastpage{pp}            
\def\correction{4}
\def\shorttitle{\it\centerline{Convergence of  the  extended Kalman filter with small and state-dependent noise}}
\def\shortauthor{\it  \centerline{Ibrahim Mbouandi Njiasse, Florent Ouabo Kamkumo,    Ralf Wunderlich}}

\usepackage{style_arxiv}
\def\webtext{} 
\def\doitext{}   

\usepackage{amsthm}
\theoremstyle{plain}
\newtheorem{theorem}{Theorem}[section]
\newtheorem{corollary}[theorem]{Corollary}
\newtheorem{lemma}[theorem]{Lemma}
\newtheorem{defination}[theorem]{Definition}
\newtheorem{remark}[theorem]{Remark}
\newtheorem{proposition}[theorem]{Proposition}

\theoremstyle{definition}

\usepackage{titlesec}
\renewcommand\thesection{\arabic{section}}

\usepackage{indentfirst}
\title{\Large \bf \boldmath\ \\ Convergence of  the  extended Kalman filter with small and state-dependent noise}   

\author{\large Ibrahim Mbouandi Njiasse$^{1}$, Florent Ouabo Kamkumo$^{1}$,    Ralf Wunderlich$^{1,\ast}$} 

\date{}

\usepackage{graphicx}
\usepackage{bm}
\graphicspath{{./}}
\usepackage[utf8]{inputenc}
\usepackage{dsfont}
\usepackage{subcaption}

\usepackage{color}
\usepackage[hyphens]{url}
\usepackage{enumerate} 
\usepackage{booktabs}
\usepackage{tabularx}
\usepackage{ltablex}
\usepackage{mathptmx}
\usepackage{amssymb,amsmath,latexsym} 
\usepackage{mathtools}
\usepackage{longtable} 
\usepackage{textcomp}
\usepackage{lscape}
\usepackage{bbm}
\usepackage{eurosym}
\usepackage[font=footnotesize,labelfont=bf]{caption}

\usepackage[usenames,dvipsnames,svgnames]{xcolor}

\usepackage{listings}

\usepackage{pdfpages}
\usepackage[pdftex]{hyperref}

\mathtoolsset{showonlyrefs=true}
\usepackage{scalerel}

\usepackage{import}
\usepackage{multicol} 
\usepackage{nomencl}

\usepackage[nice]{nicefrac} 

\usepackage{framed} 

\usepackage{ifthen}
\usepackage{calrsfs}
\DeclareMathAlphabet{\pazocal}{OMS}{zplm}{m}{n}
\let\mathcal\pazocal

\usepackage[section]{placeins}

\usepackage{amsmath,blkarray}

\usepackage{ae,aecompl}
\usepackage{graphicx}
\usepackage{palatino}
\usepackage[T1]{fontenc}
\usepackage{rotating}
\usepackage{epsf}
\usepackage{setspace}

\usepackage[utf8]{inputenc}
\usepackage[fulladjust]{marginnote}
\usepackage[]{todonotes} 
\usepackage[acronym,toc]{glossaries}	
\usepackage{float}

\usepackage{amsfonts}

\usepackage{enumitem}

\definecolor{myred}{rgb}{0.8,0,0}  

\newcommand{\mymarginpar}[1]{ \marginpar{{\tiny #1}}}

\renewcommand{\mymarginpar}[1]{}

\usepackage{soul}

\def \Prob{\mathbb{P}}
\def \E{\mathbb{E}}
\def \R{\mathbb{R}}               
\def \N{\mathbb{N}}               

\def \geqm{\succeq}
\def \leqm{\preceq}

\def \gm{\succ}

\def \sig{Y}
\def \obs{Z}
\def \covm{Q^\eps}
\def \qcovm{Q}
\def \mean{M}
\def \coDifS{g}
\def \coDifO{\ell}
\def \x{\mathcal{X}}
\def \K{\mathcal{K}}

\def \a{\eta}

\definecolor{mygreen}{rgb}{0.0,0.5,0}

\newcommand{\eps}{\varepsilon}

\renewcommand{\P}{\Prob}

\newcommand{\condmean}{M}  
 
\newcommand{\Var}{\mathrm{Var}}

\newcommand{\af}{ {\Phi}}

\allowdisplaybreaks

\def \trace{\operatorname{tr}}
\def \det{\operatorname{det}}
\def \rank{\operatorname{rank}}
\def \nullspace{\operatorname{nullspace}}

\begin{document}

\maketitle

\thispagestyle{first}
\renewcommand{\thefootnote}{\fnsymbol{footnote}}

\footnotetext{\hspace*{-5mm} \begin{tabular}{@{}r@{}p{13.4cm}@{}}
& 
\\ 
$^{\ast}$ & Corresponding author\\
\end{tabular}}


\begin{center} \it
$^{1}$Brandenburg University of Technology Cottbus-Senftenberg, Institute of Mathematics, \\P.O. Box 101344, 03013 Cottbus, Germany\\

\vskip 4.5mm

Email: Ibrahim.MbouandiNjiasse@b-tu.de, ~ Florent.OuaboKamkumo@b-tu.de, ~ralf.wunderlich@b-tu.de \\

\end{center}

	\begin{abstract}
		Nonlinear filtering problems are encountered in many applications,  and one solution approach is the extended Kalman filter, which is not always convergent. Therefore, it is crucial to identify conditions under which the extended Kalman filter provides accurate approximations. This paper generalizes two significant results  of Picard (1991) on the efficiency of the continuous-time extended Kalman filter for a filtering system with small noise, to a more general setting where the observation noise may be  state-dependent  but does not allow signal reconstruction from the quadratic variation of the observation process as  for example in epidemic models. First, we show that if the drift of the signal process and the observation process becomes nearly linear  when the parameter $\epsilon$, which scales the diffusion coefficients, approaches zero, and the drift coefficient of the observation process is strongly injective,  then the estimation error is of the order of $\sqrt{\epsilon}$. We then establish conditions under which the impact of the initial filtering error decays exponentially fast.

\vskip 4.5mm
\nd \begin{tabular}{@{}l@{ }p{10.1cm}} {\bf Keywords } &
Nonlinear filtering, small noise, state-dependent noise, extended Kalman filter, error estimate
\end{tabular}

\vskip 4.5mm

\nd  \begin{tabular}{@{}l@{ }p{10.1cm}}{\bf 2020 Mathematics Subject Classification} ~&
60G35, 62M05, 62M15, 93Ell
\end{tabular}
\end{abstract}

\baselineskip 4.5mm

	\section{Introduction}
	
	This research investigates a nonlinear filtering problem in the presence of state-dependent noise.   Many physical phenomena can best be described by nonlinear stochastic differential equations (SDEs), in which not only the drift coefficient but also the diffusion coefficient depends on the state of the system.
	These characteristics are essential for accurately modeling and understanding such systems. Filtering problems associated with these phenomena are nonstandard and necessitate alternative methods to the standard Kalman filtering approach. For instance, estimating the angular procession of a rotating spacecraft requires precise nonlinear filtering techniques to account for the state-dependent variations in noise. Similarly, the design of phase-locked loops, which are essential in communication systems to maintain signal synchronization, also requires these advanced filtering methods \cite{1099828}. Another application from mathematical epidemiology is considered below   in Section \ref{sec:epidemic_example}. These examples highlight the practical significance and wide-ranging applications of our study. 
	
	The general setting of the stochastic filtering problem can be expressed in terms of the following coupled system of SDEs on $[0,\infty)$
	\begin{align} \label{FiltPb}	
		\begin{split}
			d \sig (t) & =  f(t,\sig(t),\mathcal{\obs}(t) )dt + \sqrt{\varepsilon}\sigma(t,\sig(t),\mathcal{\obs}(t) )dW^{(1)}(t) + \sqrt{\varepsilon} \coDifS (t,\sig(t),\mathcal{\obs}(t) )dW^{(2)}(t), \\[.5ex]
			d\obs(t) & =  h(t,\sig(t),\mathcal{\obs}(t) )dt + \sqrt{\varepsilon}\coDifO(t,\sig(t),\mathcal{\obs}(t) )dW^{(2)}(t),
		\end{split}
	\end{align}
	\noindent with initial values $(\sig(0),\obs(0))=(y_0,z_0)$.
	Here,  $\sig(t) \in \R^n$ denotes the  hidden signal process   and $\obs(t) \in \R^d $ the  observation  process at time $t\geq 0$. Further,  $\mathcal{\obs}(t) =\{ \obs_s, s \leq t \}$   denotes  the  path of  the observation  process up to time   $t$.  Finally,  $\varepsilon>0 $ is a small  constant that scales the diffusion coefficients, and  $W^{(i)},~ i=1,2$, are two independent standard  Brownian motions in $\R^{d_1}$and $\R^{d_2}$ respectively (with $n,d,d_1,d_2$ being positive integers).

	This model is quite  similar to the  model in  \cite{picard1991efficiency}. However,    we address a more general setting by allowing  the diffusion coefficient  of the observation process to be non-constant and state-dependent. For such a model, when the quadratic variation of the observation process is informative and can be computed in  practice, it can be possible to   reconstruct the signal under some injectivity conditions.  However, this reconstruction is often not possible, as in the case of epidemic models, see Section~\ref{sec:epidemic_example},  for which the injectivity conditions are not fulfilled. 	It therefore remains an open question, which we investigate in this paper, whether the results of \cite{picard1991efficiency} on the efficiency of the extended Kalman filter can be transferred to the  generalized setting in \eqref{FiltPb}.

	\paragraph{Literature review}  
	Filtering problems with state-dependent diffusion coefficients in the SDE of the observation process 
	are for instance studied in   \cite{mclane1969optimal}. The author  proposed an  optimal linear filter  by  restricting the filter structure  to  the linear class  with  respect to  the observation.   Further references are \cite{takeuchi1981least, takeuchi1985least}, which provided  a more accurate approximation for the least-square optimal  estimate  as well as the convergence property of that  estimate. More recently,  \cite{joannides1995nonlinear, joannides1997nonlinear} have considered nonlinear filtering with perfect observation  and noninformative quadratic variation, which helps design robust and efficient algorithms in the case of filtering with small observation noise.  As mentioned above, the case  of state-dependent and small noise for Markovian diffusion has been investigated by \cite{picard1993estimation} and he  suggested a  stochastic algorithm  for the approximation of the optimal estimate. 
	
	On the other hand, several authors have considered the nonlinear filtering problem with small noise. Most of them  are  interested in finding efficient asymptotic  approximate filters, see \cite{katzur1984asymptotic1, katzur1984asymptotic2}.  \cite{picard1991efficiency} proved some convergence results on  the extended Kalman filter for filtering problems with constant and small diffusion coefficient in the observation process dynamic. More recently, \cite{kutoyants2025extended} considered such a filtering problem with small observation noise, and designed an adaptive extended Kalman filter algorithm that estimates both an unknown parameter and the unobservable state.

	\paragraph{Our contribution}
	The main contributions of this paper are the generalization of two important results on error estimates of the extended Kalman filter from \cite{picard1991efficiency}.  Under certain suitable assumptions, we first show that the normalized filter error for the generalized model \eqref{FiltPb} is of the order of $\sqrt{\eps}$. Second,  we show that the filtering error is bounded over time, and the impact of the initial error on the filtering error vanishes exponentially fast.

	\paragraph{ Paper outline}  In Section 2, we introduce a motivating example  from mathematical epidemiology, then Section 3 describes the derivation of the extended Kalman filter for \eqref{FiltPb}, and presents some relevant definitions related to dynamical systems. Section 4 focuses on  the estimation error of the extended Kalman filter, and we prove that under some assumptions, the filtering error is of order $\sqrt{\eps}$. Finally, in Section 5, we show that the filtering error is bounded, and the initial error is forgotten exponentially fast when certain detectability conditions are met. The Appendix collects auxiliary results needed for the proofs of our theorems.
	
	\section{Motivating example from mathematical epidemiology}\label{sec:epidemic_example}

    \begin{figure}[htbp]
    \centering   
    \includegraphics[width=4in]{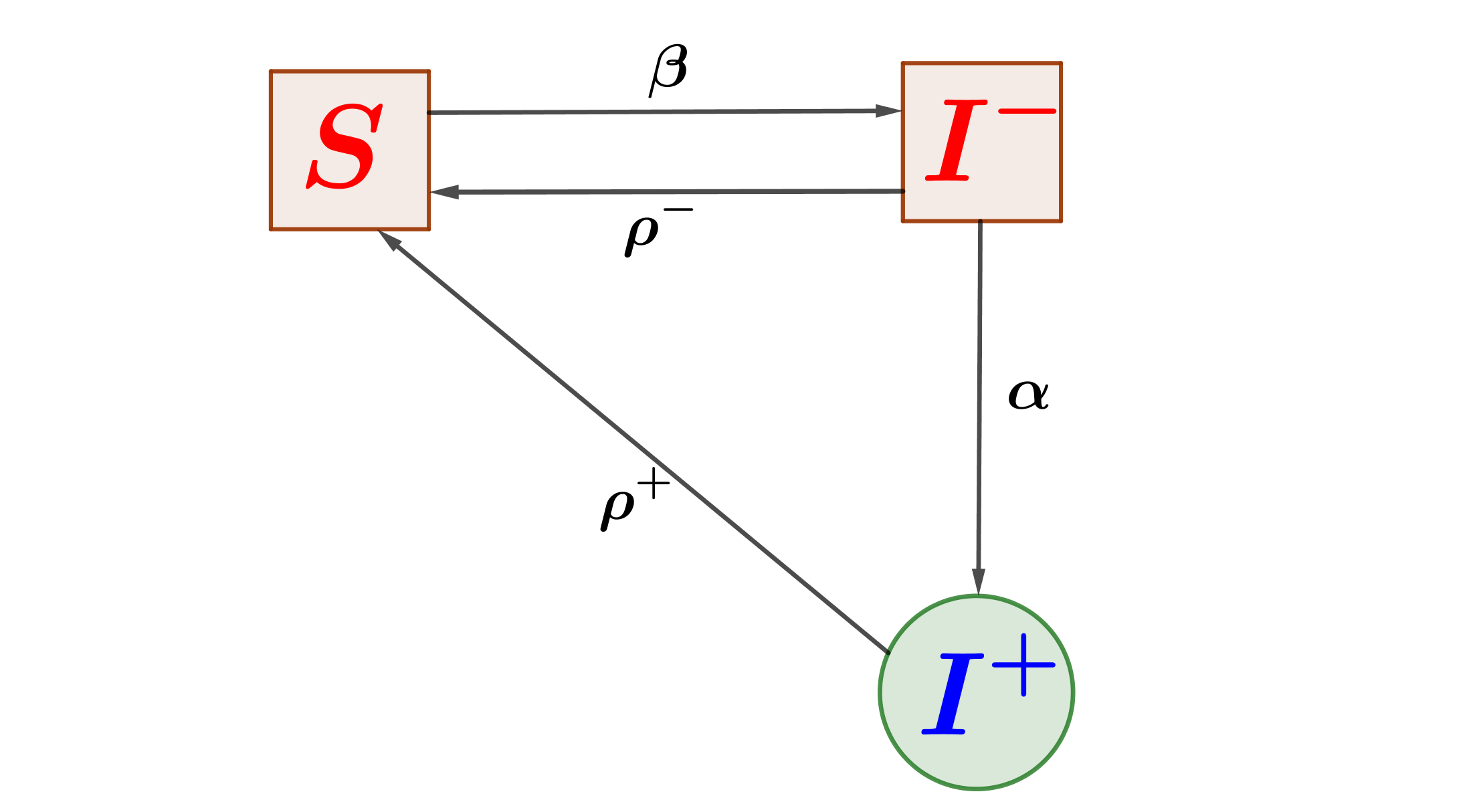}
    \caption{Flowchart of the $SI^{\pm}S$ compartmental model. }			\label{figcompartmtSIpm}
\end{figure}

	In this section, we introduce a filtering system from mathematical epidemiology that is related to stochastic epidemic models with partial information and can be approximated  by the filter system specified in \eqref{FiltPb}.  Let   us consider the $SI^\pm S$ epidemic, modeling the spread of an infectious disease in which a population is divided into three compartments according to the health state of individuals:  susceptible ($S$), undetected infected ($I^-$), and detected infected ($I^+$).  The associated flowchart is  depicted in Figure \ref{figcompartmtSIpm}.		
	Individuals that are susceptible to the disease form the compartment $S$. By contact to infectious individuals they become infected and transition with rate $\beta>0$ to the compartment $I^-$ of undetected infected individuals. Here,   presymptomatic or asymptomatic individuals and those who are not aware about their infection are collected. All of them are considered infectious and are not in quarantine. 
	The health authorities test the population for the infectious disease with a test rate $\alpha>0$. People from group $I^-$ test positive and move to group $I^+$, which consists of confirmed  or detected infected individuals who are in quarantine and can no longer infect susceptible individuals. Both detected  and undetected infected individuals recover from the infection with recovery rates  $\rho^+,\rho^->0$, and return to the susceptible compartment $S$.  A more detailed description of such models can be found in \cite{kamkumo2025estimating,kamkumo2025stochastic,ouabo2025phd}, \cite{njiasse2025stoch,mbouandi2025phd}.
	
	We denote the number of individuals in the three compartments at time $t\ge0$ by $S(t),I^-(t),I^+(t)$. The total population size  is assumed to be constant and denoted by $N$. Then for all $t\ge 0$ it holds $S(t) + I^-(t) + I^+(t) = N$.
	Since the number of confirmed infected individuals in compartment $I^+$ is known to health authorities due to monitoring through testing, $I^+(t)$ can be treated as an observable variable. However, the size of the compartments $S$ and $I^-$ is not directly observable, as the infection status of their individuals is unknown. Therefore, $S(t)$ and $I^-(t)$ must be treated as hidden or unobservable variables that can only be estimated based on observations of $I^-$.

	The stochastic dynamics of such a compartmental model can be derived using a continuous-time Markov chain approach as described in \cite{BrittonPardoux2019} and in \cite{kamkumo2025stochastic}.    
	For large population sizes $N$, diffusion approximations can be derived for the state vector $X$, which contains the relative subpopulation sizes.  These are stochastic differential equations of the  form				
	\begin{align} \label{dynmxbar}
		dX(t)&=\overline{F}(X(t))dt~+~\dfrac{1}{\sqrt{N}}\overline{\sigma}(X(t))dW(t),~ ~X(0)=x_0, 
	\end{align}
	where  coefficients $ \overline{F}$, $\overline{\sigma}$  are nonlinear functions and $W$  is a vector of independent standard Brownian motions. 
	
	In the case of the above $SI^\pm S$ epidemic model, we can work with the reduced state $X=(X_1,X_2)^\top = (\overline{I}^-,\overline{I}^+)^\top  = N^{-1}({I}^-,{I}^+)^\top $.  Since we assume a constant population size, the proportion of susceptible individuals $\overline S=S/N$ can be removed from the state, as it can be  derived from   $\overline{S}=1-\overline{I}^+-\overline{I}^-$.
	If the state $X$ is divided into the hidden component $Y= \overline I^-$ and the observable component $Z= \overline I^+$, and the notation
	$Y(t)=y,~ Z(t)=z$ and $\overline S(t)=s$ is used, the expressions of the functions  $\overline{F}$ and $\overline{\sigma}$ are as follows, see \cite{mbouandi2025phd,kamkumo2025stochastic}:		 
	\begin{align} \label{Fbar}
		\overline{F}\left( \begin{array}{l}
			y\\
			z
		\end{array}\right) =	\left(  \begin{array}{l}
			\beta s y- (\alpha+ \rho^-)y\\ [1ex] 
			\alpha y- \rho^+ z 	\end{array} \right),	
		~~
		\overline{\sigma}\left( \begin{array}{l}
			y\\
			z
		\end{array}\right) =  \left(  \begin{array}{c@{\hspace{4mm}}c@{\hspace{4mm}} c @{\hspace{4mm}} c}
			\sqrt{	\beta  s y}& - \sqrt{ \rho^-y} &  - \sqrt{\alpha y}& 0 \\ [1ex] 
			0 &0&	\sqrt{\alpha y} & - \sqrt{\rho^+ z} 	\end{array}		
		\right),
	\end{align}			
	and $W=\left({W_{1}}, \ldots,{W_{4}} \right)^\top$  with independent standard Brownian motions $W_1,\ldots,W_4$. 
	
	The expression $sy$ or, after replacing the variable $s$ with $1-y-z$, the term $(1-y-z)y$ leads to a quadratic nonlinearity in this dynamic, in particular to a nonlinear dependence of the drift coefficient on the signal $y$. This makes the filter problem nonlinear, and in this form it does not meet the technical conditions required for the subsequent analysis. Therefore, we first divide the time interval $[0,T]$ into  $N_t\in\N$ 
	subintervals of length $\Delta t=T/N_t$ and grid points $t_i=i\Delta t, i=0,\ldots,N_t$. 
	Second,  we are going to approximate the  dynamics of the state $X=(Y,Z)^\top$ on the $N_t$ small time intervals $[t_i, t_{i+1} ), ~ i=0, \ldots, N_t-1$, by freezing the variable $s$ to its value at the left endpoint of the interval $s_i=\overline{S}(t_i)$.   This is motivated by the fact that 
	in epidemic models, the proportion of susceptible individuals is usually much larger than the proportions in the other compartments. Therefore, compared to the proportions in the rather ``small'' compartments $I^-$ and $I^+$, it varies much more slowly and can be approximated well by a constant, in contrast to the latter variables, which are of greatest interest in epidemic models.
	
	This yields a filtering system on the interval $[t_i, t_{i+1} ), ~ i=0, \ldots, N_t-1$  that has linear drift coefficients and nonlinear diffusion coefficients. Moreover, this filtering system has  exactly the structure of the system  \eqref{FiltPb}, with  $\eps = {1}/{N}$ and  the following coefficient functions:      
	\begin{align} 
		f(t, y,z) ~=&~	  
		\beta s_i y- (\alpha+ \rho^-)y,                  &h(t, y,z) &=  ~ \alpha y- \rho^+ z,\\ 
		\sigma (t,y,z) ~ =&~ \left(  \begin{array}{cc}
			\sqrt{	\beta  s_i y},& - \sqrt{ \rho^-y} \end{array} \right), 
		&g (t,y,z) & =~ \left( \begin{array}{cc} - \sqrt{\alpha y},& 0
		\end{array}\right), \\ 
		\ell (t,y,z) ~ =&~ \left(  \begin{array}{cc}
			\sqrt{\alpha y}, & - \sqrt{\rho^+ z} 	\end{array} \right),
	\end{align}
	with $W^{(1)}=(W_1,W_2)^\top$  and  $W^{(2)}=(W_3,W_4)^\top$.

	\section{Extended Kalman filter and some general concepts}
	
	\subsection{Extended Kalman filter}
	The Kalman filtering procedure is well known  and has been successfully applied  in many domains.  In its standard form, however, it is necessary that, in the SDEs of the filtering system, the drift coefficients be linear in the signal and the diffusion coefficients be independent of the signal.
	Some attempts have been made to apply the core idea of the Kalman filter to nonlinear models. One of the most prominent attempts is the so-called extended Kalman filter. 
	
	 Let us fix a filtered probability spaces, denoted by  $\left( \Omega, \mathcal{F}, \mathbb{F}, \mathbb{P} \right)$.  The filtration $\mathbb{F}=(\mathcal{F}_t)_{t\ge 0}$ with $\mathcal{F}_s\subset \mathcal{F}_t\subset \mathcal{F}$ for $0\le s<t$, is such that all stochastic processes appearing in throughout this work are adapted with respect to it.
		The extended Kalman filter aims to  approximate the mean-square optimal estimate of the signal given the observations which is known to be the projection of the signal process onto the the filtration $\mathbb{F}^Z \subset \mathbb{F}  $   generated by the observation process $(Z(t))_{t\geq 0}$, i.e., $\mathbb{F}^Z=(\mathcal{F}^Z_t)_{ t \geq 0}$  with the $\sigma$-algebras $\mathcal{F}^Z_t=\sigma\{Z(s), s\le t\}$.  This estimate is given by the the conditional mean $\widehat{\condmean}(t)= \mathbb{E}[ Y(t)\vert \mathcal{F}^Z_t ]$. The  associated estimation error is measured by the conditional covariance matrix $ \Var[ Y(t)|\mathcal{F}^Z_t] =\mathbb{E}[(Y(t)- \widehat{\condmean}(t))(Y(t)-\widehat{\condmean}(t))^\top \vert \mathcal{F}^Z_t ].$      
	This extended approach  was first introduced by \cite{gelb1974applied}, and an empirical justification  was proposed by \cite{pardoux1991filtrage}.  This  approach  is  described in many books, for instance \cite{bain2008fundamentals}  for  more details.

	\begin{remark}
		In the following, many functions will have three variables: time $t$, signal $\sig(t)$, and observation path $\mathcal{\obs}(t)$, as already seen in the coefficients of the SDEs in \eqref{FiltPb}. For the sake of brevity, however, we will omit the dependence on the observation path $\mathcal{\obs}(t)$ in some expressions, but keep this dependence in mind.
	\end{remark}

	Let us revisit the main step of the   derivation of  the  extended  Kalman filter   for the nonlinear  filtering  problem \eqref{FiltPb}. 
	We denote by $\overline{\sig}$ the solution of the ordinary differential equation associated with the SDE for the signal process $Y$ in \eqref{FiltPb}, and obtained by removing the diffusion term, i.e. 
	\begin{equation}
		\dfrac{d \overline{\sig}(t)}{dt}= f(t,\overline{\sig}(t) ), ~~ ~~\overline{\sig}(0)= y_0.
	\end{equation}	
	In  a  small  interval of  time  $[0, \Delta )$, with $\Delta>0$,  we  can perform the  Taylor-like  expansion   of the nonlinear drift coefficients of  the filtering  system  \eqref{FiltPb} around the solution $\overline{\sig}$, and replace in the diffusion coefficients $\sig$ with $\overline{\sig}$.
	Then, we   obtain  the following linearized  filtering system
	
	\quad $
	\begin{array}{r@{~}cc@{~}l}
		d\sig(t) & = &  \left( \nabla_y f(t,\overline{\sig}(t) )(\sig(t)-\overline{\sig}(t))+f(t,\overline{\sig}(t) ) \right) dt &+ \sqrt{\eps}\sigma(t,\overline{\sig}(t) )dW^{(1)}(t)\\ \\
		& &&+ \sqrt{\eps}\coDifS(t,\overline{\sig}(t) )dW^{(2)}(t) ,\\\\
		d\obs(t) & = & \left( \nabla_yh(t,\overline{\sig}(t) )(\sig(t)-\overline{\sig}(t))+h(t,\overline{\sig}(t) ) \right)dt &+ \sqrt{\eps}\coDifO(t,\overline{\sig}(t))dW^{(2)}(t),
	\end{array}
	$\\
	
	\noindent  where  $\nabla_yf$   and  $\nabla_yh$    denote the  gradients  of $f$ and  $h$ with  respect to  the signal argument $\sig$,  respectively. 
	In  a  small  interval of time,  the drift coefficients  of the SDEs in the original system can be assumed to be nearly linear in the signal $\sig$, such that the above linearized SDEs provide a good approximation. This setting is then suitable to  use the results from \cite[ Chapter 12]{LiptserShiryaevVolII2001} on conditionally  Gaussian process.  By applying    it, we    obtain  approximations of    the  conditional  mean $\widehat{\sig} $ and  the  conditional  variance $\covm$ by the corresponding filter processes for  the linearized problem      that   satisfy,  see \cite[Theorem 12.66]{LiptserShiryaevVolII2001} 
	\begin{align}
			\widehat{\sig}(t) ~=\widehat{\sig}(0) &+ \int_{0}^{t}\left( \nabla_y f(s,\overline{\sig}(s) )(\widehat{\sig}(s)-\overline{\sig}(s))+f(s,\overline{\sig}(s) ) \right)ds \nonumber \\
			&+ \int_{0}^{t}G(s)\big[d\obs(s)-(h(s,\widehat{\sig}(s))-\nabla_y h(s,\overline{\sig}(s))\overline{\sig}(s)+ \nabla_yh(s,\overline{\sig}(s))\widehat{\sig}(s))ds\big] ,
		\end{align}
	where \begin{equation}
		G(t)=\big[\eps \coDifS \coDifO^\top (t,\overline{\sig}(t)) + \covm(t) \nabla_yh^\top (t,\overline{\sig}(t))\big]\big(\eps \coDifO \coDifO^\top \big)^{-1} (t,\overline{\sig}(t)),
	\end{equation}
	and  $\covm$ satisfies the Riccati differential equation 
	\begin{align}
		\dot{\covm}(t)~=&- \big[ \eps\coDifS \coDifO^\top + \covm(t) \nabla_yh^\top\big]\big(\eps \coDifO\coDifO^\top\big)^{-1}\big[\eps \coDifS \coDifO^\top + \covm(t) \nabla_y h^\top\big]^\top(t,\overline{\sig}(t))\nonumber \\	
		&+	\nabla_y f(t,\overline{\sig}(t))\covm(t) + \covm(t)\nabla_y f^\top(t,\overline{\sig}(t)) + \eps (\sigma\sigma^\top + \coDifS\coDifS^\top) (t,\overline{\sig}(t)).	
		\label{Ricatti1}
	\end{align}
	Note that this is an ordinary differential equation with random coefficients, as these depend on the observation path $\mathcal{\obs}(t)$, which has been omitted in the notation.
	By this  method, it is possible to  construct a mapping  which  relates   any  observable  process of linearization such as $ \overline{\sig}$ to  the  corresponding conditional mean of  the filter. Then,  the  extended  Kalman  filter by definition is   the  fixed  point  for  this  mapping, see \cite{bain2008fundamentals}. Hence, we  deduce  that  the extended  Kalman  filter for the nonlinear filtering problem \eqref{FiltPb}  is given by the conditional mean  $\mean$  that satisfies	
	\begin{align}
		M(t)~=& M(0) + \int_{0}^{t} f(s,M(s))ds 
		+ \int_{0}^{t}G(s)\big[d\obs(s)-h(s,M(s))ds\big], 
	\end{align}
	where \begin{equation}G(t)=\big[\varepsilon \coDifS \coDifO ^\top(t,M(t)) + \covm(t) \nabla_y h^\top(t,M(t))\big]\big(\varepsilon \coDifO \coDifO^\top\big)^{-1} (t,M(t)),
	\end{equation}
	and the conditional covariance matrix  with dynamics
	\begin{align}
		\dot{\covm}(t)=&~ - \big[ \varepsilon\coDifS \coDifO^\top + \covm(t)\nabla_y h^\top\big]\big(\varepsilon \coDifO\coDifO^\top\big)^{-1}\big[\varepsilon \coDifS \coDifO^\top + \covm(t) \nabla_y h^\top\big]^\top(t,M(t)) \\
		&~  +	\nabla_y f(t,M(t))\covm(t) + \covm(t)\nabla_y f^\top(t,M(t)) + \varepsilon (\sigma\sigma^\top + \coDifS\coDifS^\top) (t,M(t)).
		\label{Ricatti2}
	\end{align}
	
	In summary, we now formally define what will be referred to below as the extended Kalman filter for the filter system \eqref{FiltPb}. For the sake of simpler equations, we introduce the scaled covariance matrix $ \qcovm=\dfrac{1}{\eps} \covm$.
	
	\begin{defination}
		Let $(M(t))_{t \geq 0}$  and  $(\covm(t))_{t \geq 0}$ be  two  observable  processes  with  values respectively   in  $\mathbb{R}^n$ 
		and the set of positive-definite symmetric   matrices of order $n$. A process which  at  any  time  $t$ is Gaussian   with mean $M(t)$  and  covariance matrix $ \covm(t) =\varepsilon \qcovm (t)$ will  be  said  to  be an  extended  Kalman  filter  for \eqref{FiltPb} if   $M(t)$  is  solution of 	
		\begin{equation}\label{EKF1}
			M(t) = M(0) + \int_{0}^{t} f(s,M(s) )ds +
			\int_{0}^{t}G(s)\big[d\obs(s)-h(s,M(s))ds\big], 
		\end{equation}
		where 
		\begin{equation} G(t)=\big[ \coDifS \coDifO^\top (t,M(t)) +\qcovm(t)  \nabla_yh^\top(t,M(t))\big]\big( \coDifO\coDifO^\top\big)^{-1} (t,M(t)), 
			\hspace*{1.5em}	
		\end{equation}	
		\noindent and  $\qcovm(t)$ is solution  of the  Riccati  equation		
		\begin{align} 
			\begin{split}		
				\dot{\qcovm}(t)= & - \qcovm(t)[ \nabla_y h^\top (\coDifO\coDifO^\top)^{-1}  \nabla_yh](t,M(t)) \qcovm(t)
				\\	&
				+ \big[ \nabla_y f - \coDifS \coDifO^\top(\coDifO\coDifO^\top)^{-1} \nabla_y h\big](t,M(t))\qcovm(t) \\
				&+ \qcovm(t)\big[ \nabla_y f - \coDifS \coDifO^\top(\coDifO\coDifO^\top)^{-1}  \nabla_y h\big]^{\top}(t,M(t))  
				+\af(t,M(t)), \label{EKF2}			
			\end{split}	 
		\end{align}
		where  the absolute term $\af$ is given by 
		\begin{equation}
			\af(t,M(t))=\sigma\sigma^\top(t,M(t)) + \coDifS ( I-\coDifO^\top(\coDifO\coDifO^\top)^{-1}\coDifO)\coDifS^\top (t,M(t)).
			\hspace*{5em}		
		\end{equation}		
	\end{defination}

	\subsection{Some definitions}\label{sec_definitions}
	We will now introduce some key concepts from this study and recall  the notation of the underlying filtered probability space $\left( \Omega, \mathcal{F}, \mathbb{F}, \mathbb{P} \right)$. The filtration $\mathbb{F}$ is such that all processes defined in this article are adapted with respect to it.
	We will use the expression a \textit{``family of''} function or process when this function or process depends on $\varepsilon$.  It is assumed that the family of functions $f,~ h,~ \sigma,~ g, ~ \ell$  are measurable and locally bounded.  Throughout the paper,  for a vector $x$, we denote by $\vert x \vert$ the Euclidean norm of $x$, and for a matrix $A$, the matrix norm   $\vert A \vert$ is considered to  be  the supremum of  $ \vert Ax\vert$ over unit vectors $x$.
	
	\begin{defination}[Observable  Process]	
		Let  $\mathbb{F}^Z\subset \mathbb{F} $ be  the observable  filtration  generated by the observation process  $(Z(t))_{t\geq 0}$. We will  call an   \textbf{observable  process} any $\mathbb{F}^Z$-adapted  process.   
	\end{defination}
	From the above definition, it follows that
	any observable process is adapted with respect to  the natural filtration   $\mathbb{F}^Z$.
	
	\begin{defination}
		Let $\varphi:[0,\infty)\times \R^n \to \R^n,$  be  a  family  of   functions and $x_1, x_2 \in \R^n $. 
			\begin{enumerate}
				\item 
				$\varphi$  is said to be  \textbf{almost  linear} if  there  exists  a family  of matrix-valued observable processes $F:[0,\infty)\to \R^{n\times n}$  such  that 				
				\begin{equation}
					\vert \varphi(t,x_1)- \varphi(t,x_2) - F(t) (x_1-x_2) \vert  \leq \mu^{\eps} \vert x_1-x_2 \vert,
				\end{equation}
				for some family of  numbers  $\mu=(\mu^{\eps})_{\eps>0}$ converging to zero as $\eps \to 0$. \\ The  process $F$  will be  called an  \textbf{almost  derivative} of $f$.
				
				\item The  function  $\varphi$  will be  said  to  be \textbf{strongly  injective} if 
				\begin{equation}
					\vert \varphi (t,x_1)- \varphi(t,x_2)  \vert  \geq c \vert x_1-x_2 \vert,
				\end{equation}
				for  some $c>0$.
			\end{enumerate}
		\end{defination}
		
		\begin{defination} Let $L^q$ be the space of random variables with  finite moment of order $~q \geq 1$. 
			The \textbf{norm of a random variable} $X$ in $L^q,~q \geq 1$,  is denoted by  $\left\vert X \right\vert_{q}$  and is  defined  by  $\left\vert X \right\vert_{q}=\Big(\E[|X|^q]\Big)^{1/q}$.
		\end{defination}

		\begin{defination}
			A family of processes $(\x^\eps(t))_{t\ge 0}$ is said to be \textbf{bounded in $L^{\infty-}$} if for any $q \in (0,\infty)$, there exists $\eps_q > 0$ such that $\vert \x^\eps(t) \vert_q$  is bounded uniformly in $t \geq 0$  and $0 < \eps < \eps_q$.
		\end{defination}
		\begin{defination}[Stochastic Process of Order $\eps^\kappa$]
			A family of stochastic processes $ \x(t)=(\x^\eps(t))_{t\ge 0}$, $\eps>0$ is said to be \textbf{ of order $\eps^\kappa$} 
			for some $\kappa>0$  if for any $q\in [1,\infty)$ there exist   constants $\eps_q>0, ~ C_q>0$, 
			such that 			
			$$\eps^{-\kappa} \left\vert\x^\eps(t) \right\vert_{q} \le C_q   ~~\text{for all}~~ (t,\eps)\in[0,\infty)\times (0,\eps_q).	$$			
		Similarly, a family of random variables  $ X=(X^\eps)$, $\eps>0$ is said to be of order $\eps^\kappa$ 
			for some $\kappa>0$  if for any $q\in [1,\infty)$ there exist   constants $\eps_q>0,~ C_q>0$, 
			such that			
			$$\eps^{-\kappa} \left\vert{X^\eps} \right\vert_{q} \le C_q   ~~\text{for all}~~ \eps\in (0,\eps_q).	$$
		\end{defination}
       \begin{remark}
		To avoid cluttering the notation, we usually suppress the superscript index $\eps$ in the notation of families of function, stochastic processes, and random variables, but keep this dependency in mind.
        \end{remark}

		In  the study of dynamical systems, stability properties are  powerful to characterize the qualitative behavior of a solution. We introduce two concepts of stability  for a family of matrices. 
		
		\begin{defination}
			Let  $(A(t))_{t\ge 0}$  be  a family  of  measurable  locally  bounded processes with  values in the class of  $n \times n $  matrices, let $\zeta $  be  the matrix-valued  solution of \begin{equation}
				\dot{\zeta}(t)= A(t) \zeta(t),~~~\zeta(0)= I.
			\end{equation}
			
			\noindent	 We will  say  that $(A(t))_{t\ge 0}$  is  \textbf{exponentially  stable} if there  exist  some positive  constants  $C$  and $c$    such  that for $s \leq t$,
			\begin{equation} \label{expstab}
				\vert \zeta(t)\vert\leq C e^{ (- c(t-s))} \vert \zeta(s) \vert.
			\end{equation}
		\end{defination}
		
		\begin{remark}
			In case of  constant matrix $A$,  it is well-known that $A$ is exponentially stable if and only if its eigenvalues have negative real part.
		\end{remark}
		
		\begin{defination}
			Let $A, B \in\mathcal{M}_n$, where $\mathcal{M}_n$ denotes the class of $n \times n$ real matrices.\\ We write $A\geqm  B$ if $A$ and $B$  are symmetric and
			$A -B$ is positive semi-definite.\\ We write $A \succ B$ if $A$ and $B$ are symmetric and $A -B$
			is positive definite.
		\end{defination}
		
		\begin{defination}\label{Kk_stable}
			Consider a family $(\K(t))_{t\ge 0}$ of absolutely continuous adapted (w.r.t. the observable filtration) processes with values in the class of  symmetric positive definite $n \times n $ matrices and a family $(k(t))_{t\ge 0}$ of locally bounded deterministic functions with positive values. We will say that  a family $(A(t))_{t\ge 0}$  of measurable locally bounded process with values in the class $n \times n$ matrices,  is $(\K(t), k(t))_{t\ge 0}$-stable,  or in short, that $A$ is $(\K, k)$-stable, if  for all $t\ge 0$,
			\begin{equation} \label{QKstab}
				\dot{\K}(t)\geqm A(t)\K(t) + \K(t)A^{\top}(t) + k(t)\K(t).
			\end{equation}			
		\end{defination}
		Note that the relationship between these two stability concepts will be established below in Lemma \ref{lem2}.

		\section{Error estimation   for  the extended Kalman  filter}
		When estimating a hidden signal, it is of relevance to know whether  the scale of the error between the estimate and the true signal is large or not. This is why   
		we are interested in conditions under which  the extended Kalman filter for problem \eqref{FiltPb} provides a good approximation in the case of a strongly injective drift  coefficient of the observation process. We show in  the following  theorem   that the strongly injectivity  combined  with an initial  guess of order $\sqrt{\eps}$  and a set of technical  conditions ensure that the estimation error $\sig(t)-\condmean(t) $ is of order  $\sqrt{\eps}$ at any future time point $t$. 
		\begin{theorem} \label{CvgTH}
			Let $(M(t), \covm (t))_{t \geq 0}$  be   an extended  Kalman filter for problem \eqref{FiltPb} such  that  $M$  and $ \qcovm= \frac{1}{\eps }\covm$
			satisfy  \eqref{EKF1}  and \eqref{EKF2}  respectively,  and  in  addition 
				\begin{enumerate}[ wide, labelwidth=!, labelindent=0pt]
					\renewcommand{\labelenumi}{(A\theenumi)} 				
					\item \label{cond1}  The  variable $\qcovm^{-1/2}(0) (\sig(0)-M(0))$  is  of  order $\sqrt{\eps}$.
					\item \label{cond2} The functions  $\sigma$, $\coDifS$  and $\coDifO$  are  bounded.
					\item \label{cond3} The  function $f$ and $h$  are  $\mathcal{C}^1$  and   almost  linear with  respect  to  the  signal.
					\item \label{cond4} 
					\hspace*{-3mm}
					\begin{tabular}[t]{l}The    function $h$   is   strongly injective with  respect  to  the  signal,\\ $\af= \sigma \sigma^\top - \coDifS(I-\coDifO^\top(\coDifO\coDifO^\top)^{-1} \coDifO)\coDifS^\top$   and $(\coDifO\coDifO^\top)^{-1}$ are uniformly  elliptic, and \\   the quotient   between  the largest   and  the  smallest eigenvalues  of  $\qcovm (0)$   is  bounded.
					\end{tabular}	
				\end{enumerate}
				
				\noindent	Then   $\qcovm^{-1/2}(t) (\sig(t)-M(t))$  is  of  order $\sqrt{\eps}$.
				
			\end{theorem}
			
			\begin{remark} Assumptions (A\ref{cond2})-(A\ref{cond4}) of this theorem can be checked  for the $SI^\pm S$ epidemic model  considered in Sec.~\ref{sec:epidemic_example} on each small  interval $[t_i, t_{i+1} ), ~ i=0, \ldots, N_t-1$,   under the natural assumption, that  all compartment sizes are strictly positive, i.e., the proportions $Y$ and $Z$ of nondetected and detected infected take values only in the open interval $(0,1)$. The assumption (A\ref{cond2})) is fulfilled  by the definition of the functions  $\sigma$, $\coDifS$  and $\coDifO$ which are bounded for  $y,z\in(0,1)$. The drift coefficients $f$ and $h$ in the SDEs for the hidden signal  and the observation are linear functions of $y$, since the proportion of susceptible individuals $s$ is treated as a constant. Thus  $f$ and $h$  are $\mathcal{C}^1$  functions, and  also  almost linear, hence 
            (A\ref{cond3}) is fulfilled.
				
				For assumption (A\ref{cond4}) it can be observed, that  $h$ is strongly injective since it is linear, and $\af$  and $(\ell \ell ^\top)^{-1}$  are strictly positive scalars since all compartment sizes are strictly positive. The initial matrix $\qcovm(0)$ is a scalar and can be chosen to satisfy the assumption.

			\end{remark}

			\noindent 	In  order  to  prove    Theorem  \ref{CvgTH}, we    first  establish  the  following lemma  and  then we will   show that the  theorem is  a  particular  case  of  this  lemma.

			\begin{lemma}\label{Lemma_2.2.2}
				Let $(M(t),\covm(t))_{t \geq 0}$ be an EKF satisfying assumptions (A\ref{cond1})-{(A\ref{cond4})} of Theorem \ref{CvgTH}. Suppose also that { $ \qcovm= \eps^{-1} \covm$ }		satisfies 
				\begin{equation} \label{cond1lem222}
					\qcovm(t)  \big(\nabla_y h^{\top} (\coDifO\coDifO^\top)^{-1}\nabla_yh \big )(t,M(t))\qcovm(t) +\af(t,M(t))\geqm k(t)\qcovm(t),
				\end{equation}
				for some family of deterministic positive functions $k(t)$   and the function $\af$ given in (A\ref{cond4}). Further, it is assumed that    there is a constant $C>0$ such that  
				\begin{equation}
					\label{cond2lem222} 	
					\qcovm(t) + \qcovm^{-1}(t)\leqm Ck(t)I.
				\end{equation}
				Then $\qcovm^{-1/2} (t)\big(\sig(t)-M(t)\big)$ is of order $\sqrt{\eps}$.
			\end{lemma}
			
			\noindent	The proof of  this Lemma is based on the following lemma, which is proven  in \cite{picard1991efficiency}. 
			
			\begin{lemma}\label{Lemma_1.3.2}
				Let $\widetilde{W}(t)$ be a $\mathbb{F}$-Brownian motion with values in $\mathbb{R}^{r}$, $n \in \mathbb{N}$, and let $A$ be a family of $\mathbb{F}$-adapted $(\K, k)$-stable processes where $\K$ is $\mathbb{F}$-adapted  and $k$ is  some family of deterministic positive and  locally bounded  functions. We suppose that $\x(t)$ is a family of $\mathbb{R}^{n}$-valued semi-martingales satisfying
				\begin{equation}
					d\x(t)=A(t)\x(t)dt + F(t)dt + D(t)d\widetilde{W}(t), 
				\end{equation}
				where $F(t)$ and $D(t)$ are  $\mathbb{F}$-adapted processes satisfying   for some $ \a \in [0,1/2) $ 
				\begin{align} 
					\x^{\top}(t)\K^{-1}(t)F(t) &\leq           ~   \a k(t)\x^{\top}(t)\K^{-1}(t)\x(t) + \mathcal{O}(k(t))\label{cond1lem132}, \\
					D^{\top}(t)\K^{-1}(t)D(t) &=~ \mathcal{O}(k(t)).\label{cond2lem132}
				\end{align}	
				If $\x^{\top}(0)\K^{-1}(0)\x(0)$ is bounded in $L^{q}$, 		
				then for any $q\geq 1$ and $\eps$ small enough,  the process\\ $\x^{\top}(t)\K^{-1}(t)\x(t)$ is  bounded in $L^{q}$.
			\end{lemma}
			
			\begin{proof} (of Lemma \ref{Lemma_2.2.2})				
				Let  us recall  the  dynamics of the signal $Y$ and  of the extended Kalman filter $M$ from \eqref{FiltPb} and \eqref{EKF1}, respectively.		
				 	\begin{align*}
						d\sig(t) ~=& f(t,\sig(t))dt + \sqrt{\varepsilon}\sigma(t,\sig(t))dW^{(1)}(t) + \sqrt{\varepsilon} \coDifS (t,\sig(t))dW^{(2)}(t),\\
						dM(t)~=& f(t,M(t))dt + \big[\coDifS \coDifO^{\top}(t,M(t)) + \qcovm(t)\nabla_y h^{\top}(t,M(t))\big]\big(\coDifO \coDifO^{\top}\big)^{-1}\times\\
						&\hspace*{6em} \big[d\obs(t) - h(t,M(t))dt\big],\\
						~=& f(t,M(t))dt + \big[\coDifS \coDifO^{\top}(t,M(t)) + \qcovm(t)\nabla_y h^{\top}(t,M(t))\big]\big(\coDifO\coDifO^{\top}\big)^{-1}\times \\
						& \hspace*{7em} \big[h(t,\sig(t))dt  + \sqrt{\eps}\coDifO(t,\sig(t))dW^{(2)}(t) - h(t,M(t))dt\big].
				\end{align*} 	
				Thus, if we denote the difference $\sig-M $ by $\x$, its dynamics is given by  
				
                    \begin{align*}	
					d\x(t)~=& \big[f(t,\sig(t)) - f(t,M(t))\big]dt\\
					& -\big[\coDifS \coDifO^{\top}(t,M(t)) + \qcovm(t)\nabla_yh^{\top}(t,M(t))\big]\big(\coDifO\coDifO^{\top}\big)^{-1} \big[h(t,\sig(t)) - h(t,M(t))\big]dt\\
					& + \sqrt{\eps}\sigma(t,\sig(t))dW^{(1)}(t) + \sqrt{\eps}\coDifS (t,\sig(t))dW^{(2)}(t)\\
					& - \sqrt{\eps}\big[\coDifS \coDifO^{\top}(t,M(t)) + \qcovm(t)\nabla_yh^{\top}(t,M(t))\big]\big(\coDifO\coDifO^{\top}\big)^{-1}\coDifO(t,\sig(t))dW^{(2)}(t)\\
					~=& \big[f(t,\sig(t)) - f(t,M(t))\big]dt \\
					&-\big[\coDifS \coDifO^{\top}(t,M(t)) + \qcovm(t)\nabla_yh^{\top}(t,M(t))\big]\big(\coDifO\coDifO^{\top}\big)^{-1} \big[h(t,\sig(t)) - h(t,M(t))\big]dt\\
					&+ \sqrt{\eps}\sigma(t,\sig(t))dW^{(1)}(t) + \sqrt{\eps}\Big[\coDifS(t,\sig(t))\\
					&\hspace*{2em}-\big[\coDifS \coDifO^{\top}(t,M(t)) + \qcovm(t)\nabla_yh^{\top}(t,M(t))\big]\big(\coDifO\coDifO^{\top}\big)^{-1}\coDifO(t,\sig(t))\Big]dW^{(2)}(t).
				\end{align*}

				Here,	the second equality is obtained by simple expansion. Adding and subtracting the same term and rearranging, we get 
				\begin{align*}	
					d\x(t)	~=&\Big[f(t,\sig(t)) - f(t,M(t))-\Big(\nabla_yf(t,M(t))-\nabla_yf(t,M(t))\Big)\x(t) \Big]dt\\
					&- \big[\coDifS \coDifO^{\top}(t,M(t)) + \qcovm(t)\nabla_yh^{\top}(t,M(t))\big]\big(\coDifO\coDifO^{\top}\big)^{-1}\\
					&\times  \big[h(t,\sig(t)) - h(t,M(t)) - \big(\nabla_yh(t,M(t))-\nabla_yh(t,M(t))\big)\x(t) \big]dt \\
					& + \sqrt{\eps}\Big(\sigma(t,\sig(t)),\coDifS(t,\sig(t))- \\
					&\big[\coDifS \coDifO^{\top}(t,M(t)) + \qcovm(t)\nabla_yh^{\top}(t,M(t))\big]\big(\coDifO\coDifO^{\top}\big)^{-1}\coDifO(t,\sig(t))\Big) \left(\begin{array}{c} dW^{(1)}(t)\\ dW^{(2)}(t) \end{array} \right)\\
					~=& \left[\nabla_yf(t,M(t))- \big( \coDifS\coDifO^{\top}(t,M(t))+ \qcovm(t)\nabla_yh^{\top}\big)(\coDifO\coDifO^{\top})^{-1}\nabla_yh(t,M(t)) \right]\x(t) dt\\
					&  +\Big[ f(t,\sig(t)) - f(t,M(t)) -\nabla_yf(t,M(t))\x(t)  \\
					& 	- \big(\coDifS\coDifO^{\top}(t,M(t)) + \qcovm(t)\nabla_yh^{\top}(t,M(t))\big)  (\coDifO\coDifO^{\top})^{-1}  \\ 	
					&\times  \big(h(t,\sig(t)) - (h(t,M(t))-\nabla_yh(t,M(t))\x (t) )\big)\Big] dt + \sqrt{\eps}\Big(\sigma(t,\sig(t)), \\
					& \coDifS(t,\sig(t)) - \big[\coDifS \coDifO^{\top} + \qcovm (t)\nabla_yh^{\top}\big](t,M(t))\big(\coDifO\coDifO^{\top}\big)^{-1}\coDifO(t,\sig(t))\Big) \left(\begin{array}{c}  dW^{(1)}(t)\\ dW^{(2)}(t) \end{array}\right).
				\end{align*}
				\paragraph{Application of Lemma \ref{Lemma_1.3.2}} 
				 The above SDE for $\x$ can be rewritten in terms of the notation used in Lemma \ref{Lemma_1.3.2} by using  following settings  
				\begin{align*}	
					\K(t)~&=  \qcovm(t),		\\
					d\x(t)~&= A(t) \x(t) dt + F(t)dt +D(t) d\widetilde{W}(t),
				\end{align*}			
				where 			
				\begin{align}
					\label{Adef}
						A(t)~:=&  \nabla_yf(t,M(t))- \big( \coDifS\coDifO^{\top}(t,M(t))+ \qcovm(t)\nabla_yh^{\top}\big)(\coDifO\coDifO^{\top})^{-1}\nabla_yh(t,M(t)) ,\\
						F(t)~:=&  f(t,\sig(t)) - f(t,M(t)) -\nabla_yf(t,M(t))\x(t) \\
						& - \big(\coDifS\coDifO^{\top}(t,M(t)) + \qcovm(t)\nabla_yh^{\top}(t,M(t))\big)  \\ 	
						&\times (\coDifO\coDifO^{\top})^{-1}(t,M(t))    \big(h(t,\sig(t)) - (h(t,M(t))-\nabla_yh(t,M(t))\x(t) )\big)  \\
						D(t)~:=& \sqrt{\eps} \big(\sigma(t,\sig(t)),~ \coDifS(t,\sig(t))  \\
						&  - \big[\coDifS \coDifO^{\top}(t,M(t)) + \qcovm(t)\nabla_yh^{\top}(t,M(t))\big]\big(\coDifO\coDifO^{\top}\big)^{-1}\coDifO(t,\sig(t))\big),\\
						\widetilde{W}(t):=~& \left(\begin{array}{c} W^{(1)}(t)\\ W^{(2)}(t) \end{array} \right).
					\end{align}		
					We will now verify the assumptions from Lemma \ref{Lemma_1.3.2}.
					\paragraph{$A$ is $(\qcovm,k)$-stable}		
				 We recall the Riccati ODE \eqref{EKF2}	for $\qcovm$				
				\begin{align*}					
					\dot{\qcovm}(t)= &- \qcovm(t)\nabla_yh^\top(\coDifO\coDifO^\top)^{-1}\nabla_yh\qcovm(t) +\big[\nabla_yf - \coDifS \coDifO^\top(\coDifO\coDifO^\top)^{-1}\nabla_yh\big]\qcovm(t) 
					\\
					&+ \qcovm(t)\big[\nabla_yf - \coDifS \coDifO^\top(\coDifO\coDifO^\top)^{-1} \nabla_yh\big]^{ \top}  +\af(t,M(t)),
				\end{align*}
				and condition \eqref{cond1lem222} which implies  that for some family of deterministic positive functions $k$ it holds 
				 $\af(t,M(t))\geqm - \qcovm(t)\nabla_yh^\top{ (\coDifO\coDifO^\top)^{-1}}\nabla_yh\qcovm(t) + k(t)\qcovm(t)$. Substituting this in the above ODE and using the definition of $A$, see above in \eqref{Adef}, yields 				
				\begin{align*}
					\dot{\qcovm}(t)
					\geqm &- 2\qcovm(t)\nabla_yh^\top(\coDifO\coDifO^\top)^{-1}\nabla_yh\qcovm(t) +\Big[\nabla_yf -\coDifS \coDifO^\top(\coDifO\coDifO^\top)^{-1}\nabla_yh\Big]\qcovm(t)
					\\
					&+ \qcovm(t)\Big[\nabla_yf - \coDifS \coDifO^\top(\coDifO\coDifO^\top)^{-1} \nabla_y h\Big]^{\top}
					+k(t)\qcovm(t) ,\\
					 =& A(t)\qcovm(t) + \qcovm(t)A^\top(t)+k(t)\qcovm(t). \qquad                                  
				\end{align*}
				This proves that $A$ is $(\qcovm,k)$-stable.
					
				\paragraph{ Verfication of condition \eqref{cond1lem132}}
				For that, it is sufficient to show  that     for some family of functions $\overline \eta$  with values  in $[0, {1}/{4}) $, it holds 
					\begin{equation} \label{condSub}
						F^{\top}(t)\K^{-1}(t)F(t) \leq  \overline \eta k^2(t) \x^{\top}(t)\K^{-1}(t)\x(t).
					\end{equation}					
				This follows from  	an application of the Cauchy-Schwarz inequality corresponding  to  the inner product   defined by the positive definite matrix $\K^{-1}(t)$ to
						
						\begin{align*}
							\left\vert \x^{\top}(t)\K^{-1}(t)F(t) \right\vert &\leq \Big(\x^{\top}(t)\K^{-1}(t)\x(t)\Big)^{1/2}\Big(F^{\top}(t)\K^{-1}(t)F(t)\Big)^{1/2}\\
							&\leq \Big(\x^{\top}(t)\K^{-1}(t)\x(t)\Big)^{1/2}\Big( \overline \eta k^{2}(t)\x^{\top}(t)\K^{-1}(t)\x(t)\Big)^{1/2}\\
							&= (\overline \eta)^{1/2} k(t)\x^{\top}(t)\K^{-1}(t)\x(t).
						\end{align*}	
						Hence, we have the relation $\eta = \overline \eta)^{1/2}$.	 Next, we have for  $\K(t)= \qcovm(t)$
						\begin{align*}
							F^{\top}(t)\qcovm^{-1}(t)F(t)&= \Big[B
							- \left(\coDifS(t,M(t))\coDifO^{\top}(t,M(t)) - \qcovm(t)\nabla_yh^{\top}\right)(\coDifO\coDifO^{\top})^{-1}H\Big]^{\top}\\
							&~~~~~\times \qcovm^{-1}(t)\times \Big[B
							- \left(\coDifS(t,M(t))\coDifO^{\top}(t,M(t)) - \qcovm(t)\nabla_yh^{\top}\right)(\coDifO\coDifO^{\top})^{-1}H\Big],
						\end{align*}
						with 
						\begin{align*}
							B&:= f(t,\sig(t)) - f(t,M(t)) -\nabla_yf(t,M(t))(\sig(t)-M(t)), \\[0.5ex]
							H&:= h(t,\sig(t))-h(t,M(t)) -\nabla_yh(t,M(t))(\sig(t) - M(t)).
						\end{align*}
						Thus, since $f$ and $h$ are almost linear, there exists a family  of positive constants  $\overline{\mu}=(\overline{\mu}^\eps)_{\eps>0}$  converging to $0$ as $\eps \rightarrow 0$ , such that  $ |B|,|H| \leq \overline{\mu} |\sig(t)-M(t)|$. 						
						Moreover,  we have that,
						\begin{align*}
							F^{\top}(t)\qcovm^{-1}(t)F(t)=&~ B^{\top}\qcovm^{-1}B 
							\\
							&- H^{\top}(\coDifO\coDifO^{\top})^{-1}\Big(\coDifS(t,M(t))\coDifO^{\top}
							+ \qcovm(t)\nabla_yh(t,M(t))\Big)^{\top}(\coDifO\coDifO^{\top})^{-1}\qcovm^{-1}B\\ 
							&+ H^{\top}(\coDifO\coDifO^{\top})^{-1}\Big(\coDifS(t,M(t))\coDifO^{\top} + \qcovm(t)\nabla_yh(t,M(t))\Big)^{\top}(\coDifO\coDifO^{\top})^{-1}\qcovm^{-1}\\
							&\times \Big(\coDifS(t,M(t))\coDifO^{\top} + \qcovm(t)\nabla_yh(t,M(t))\Big)(\coDifO\coDifO^{\top})^{-1}H(t)\\
							&- B^{\top}\qcovm^{-1}\Big(\coDifS(t,M(t))\coDifO^{\top} + \qcovm(t)\nabla_yh(t,M(t))\Big)(\coDifO\coDifO^{\top})^{-1}H(t)				
						\end{align*}						
						The expression of $ F^{\top}(t)\qcovm^{-1}(t)F(t)$ involves  matrices  of the form $\qcovm(\coDifO\coDifO^{\top})^{-1} \qcovm^{-1}$  and $\qcovm^{-1}(\coDifO\coDifO^{\top})^{-1} \qcovm$. Using the fact  that $Q$ and $Q^{-1}$ are symmetric and positive definite, one  can show that these two matrices have the same eigenvalues as the matrix $(\coDifO\coDifO^{\top})^{-1}$.  Since $(\coDifO\coDifO^{\top})^{-1}$ is bounded, these matrices are also bounded.			
						Thus, using our assumptions   that  $\coDifS$, $\coDifO$  and  $(\coDifO\coDifO^{\top})^{-1}$ are  bounded, there exists   a family of positive constants  $\mu =(\mu^{\varepsilon})_{\eps>0}$  depending on  $ \overline{\mu}$ such that    
						\begin{align*}
							F^{\top}(t)\qcovm^{-1}(t)F(t)
							\leq &~\mu^{\varepsilon}(1+|\qcovm(t)| + |\qcovm^{-1}(t)|)|\sig(t)-M(t)|^{2}.
						\end{align*}
						In addition,   $\qcovm(t) + \qcovm^{-1}(t)\leqm C k(t) I$ implies that   $|\qcovm(t)| + |\qcovm^{-1}(t)|\leq 2 C|k(t)|$, therefore we can rewrite the last inequality as 					
						\begin{equation} \label{ineq1}
							F^{\top}(t)\qcovm^{-1}(t)F(t)\leq 2 \mu^{\varepsilon}Ck(t)(\sig(t)-M(t))^{\top}I(\sig(t)-M(t)). 
						\end{equation}			
						Then, from   $\qcovm(t) \leqm Ck(t) I$, we obtain $Ck(t)\qcovm^{-1}\geqm I$ and the  inequality \eqref{ineq1}  can  take  the following  form (for a new constant C), 
					\begin{equation}
						F^{\top}(t)\qcovm^{-1}(t)F(t)\leq \mu^{\varepsilon}Ck^{2}(t)(\sig(t)-M(t))^{\top}\qcovm^{-1}(t)(\sig(t)-M(t)). 
					\end{equation}
					The condition \eqref{condSub} is satisfied for some $\mu^{\varepsilon}\to 0$, which completes the proof of condition \eqref{cond1lem132}. 
					\paragraph{ Verfication of condition \eqref{cond2lem132}}
					This  can be deduced from the assumption \eqref{cond2lem222} as follows.
					We have that		
					$\qcovm(t) + \qcovm^{-1}(t) \leqm  Ck(t)I $.  This implies that  ~$D^{\top}(t)(\qcovm(t) + \qcovm^{-1}(t))D(t) \leqm D^{\top}(t) Ck(t)I D(t)$~ and thus 
					\begin{align}						
						~~ D^{\top}(t)\qcovm(t)D(t) + 	D^{\top}(t)\qcovm^{-1}(t)D(t)& \leqm  Ck(t)\vert D(t)\vert^{2} .
					\end{align}
					Since $\qcovm(t)$  and $\qcovm^{-1}(t)$ are  positive definite, it holds  $	D^{\top}(t)\qcovm(t)D(t)\geqm 0$  and also $	D^{\top}(t)\qcovm^{-1}(t)D(t)\geqm 0$.  In addition, $\vert D(t)\vert$ is bounded, thus 
					$D^{\top}(t)\qcovm^{-1}(t)D(t)=O(k(t))$  which proves   \eqref{cond2lem132}.
					 \paragraph{Final conclusion}
						In view of Definition \ref{Kk_stable}one can  see   that if $A$ is $(\qcovm,k)$-stable then  it is also    $(\eps \qcovm,k)$-stable for any $\eps>0$. Thus, we can apply  Lemma \ref{Lemma_1.3.2} with $\K=\eps\qcovm$ (instead of $\K=\qcovm$) and deduce that $\x^\top \mathcal{K}^{-1} \x =(\sig(t)-M(t))^\top (\eps \qcovm)^{-1} (\sig(t)-M(t))$ is bounded in $L^2$, hence    $\eps^{-1/2}|\qcovm^{-1/2}(t)(\sig(t)-M(t))|^2 \leq C$, for some $C>0$. This implies that  $\qcovm^{-1/2}(t)(\sig(t)-M(t))$ is of order $\sqrt{\varepsilon}$.
					            
				\end{proof}

			\begin{proof}   (of Theorem \ref{CvgTH})\\
				Since we  have  assumed  (A\ref{cond4}) that $\af$   and $(\coDifO\coDifO^\top)^{-1}$ are uniformly  elliptic   and   that  $h$  is  strongly injective,  we  have 				
				\begin{equation}
					\qcovm(t) \nabla_yh^\top (\coDifO\coDifO^\top)^{-1} \nabla_yh \qcovm(t) +\af(t, M(t))\geqm c(\qcovm^2(t)+ I) \geqm c\left(\qcovm(t)+\qcovm^{-1}(t)\right)\qcovm(t).
				\end{equation}			
				If  there  exists  a  family of  deterministic  positive  functions $p(t)$  such   that  $p^{-1}(t)\qcovm(t)$  is  uniformly  bounded  and  elliptic,  then  the   conditions  of  Lemma \ref{Lemma_2.2.2}  will be  satisfied with  $k(t)$   proportional  to $p(t)+p^{-1}(t)$. Namely,  if  we assume that   $p^{-1}(t)\qcovm(t)$  is  uniformly  bounded  and  elliptic, it  implies  that there  exist  two  positive  real constants $b_1$  and $b_2$ satisfying 			
				\begin{equation}
					b_1 I \leqm p^{-1}(t)\qcovm(t) \leqm  b_2 I.
				\end{equation}			
				From  the  first  part  of  this inequality we deduce that  $b_1   p(t) I \leqm \qcovm(t)$   and  from  the second  part we  get   $~\dfrac{p^{-1}(t)}{ b_2} I \leqm \qcovm(t)^{-1}$.  		
				Combining  these two inequalities, we  get \begin{equation}
					\qcovm(t)+  \qcovm(t)^{-1} \geqm  \left(b_1   p(t)+ \dfrac{p^{-1}(t)}{ b_2} \right) I \geqm b \left(  p(t)+ p^{-1}(t) \right) I,
				\end{equation} 
				for  some constant $b$.  	Thus, we  only  have  to  prove the  existence  of   such   a  family $p(t)$. 
				
				Let  $p(0)$  be   the  trace of  $\qcovm(0)$, as a  reminder  the   dynamics  of $\qcovm$  is  given by, see \eqref{EKF2},		 
				\begin{align*} 
						&\dot{\qcovm}(t) =- \qcovm(t)\Big[ \nabla_y h^\top(\coDifO\coDifO^\top)^{-1}  \nabla_yh\Big](t,M(t)) \qcovm(t)\\
						&\hspace{1em} +\Big[ \nabla_y f - \coDifS \coDifO^\top(\coDifO\coDifO^\top)^{-1} \nabla_y h\Big](t,M(t))\qcovm(t) + \qcovm(t)\Big[ \nabla_y f - \coDifS \coDifO^\top(\coDifO\coDifO^\top)^{-1}  \nabla_y h\Big]^{\top}(t,M(t))\\
						&\hspace{1em}	 +\af(t,M(t)).
					\end{align*}				
				Then, applying  the  trace  to   this  dynamic,  we  obtain 			
				\begin{align}
					\trace (\dot{\qcovm}(t))=&~\trace\left(\Big[ \nabla_y f - \coDifS \coDifO^\top(\coDifO\coDifO^\top)^{-1} \nabla_y h\Big]\qcovm(t)\right) + \trace\left(\qcovm(t)\Big[ \nabla_y f - \coDifS \coDifO^\top(\coDifO\coDifO^\top)^{-1}  \nabla_y h\Big]^{\top}\right)\nonumber\\
					&	- \trace\left( \qcovm(t)\Big[ \nabla_y h^\top(\coDifO\coDifO^\top)^{-1}  \nabla_yh\Big] \qcovm(t)\right) + \trace(\af(t,M(t))).\nonumber
				\end{align}
				Since $~\nabla_y f$, $\nabla_y h$,    $(\coDifO\coDifO^\top)^{-1}$  and $\coDifS$  are  bounded, and  using  the trace inequality \eqref{trace1},  we  have 
				\begin{align}
						\trace\big(\qcovm(t)\big[ \nabla_y f - \coDifS \coDifO^\top(\coDifO\coDifO^\top)^{-1}  \nabla_y h\big]^{\top}(t,M(t))\big)\hspace*{14em}\\
						\leq  \big\vert \trace\big(\qcovm(t)\big[ \nabla_y f - \coDifS \coDifO^\top(\coDifO\coDifO^\top)^{-1}  \nabla_y h\big]^{\top}(t,M(t))\big) \big\vert \hspace*{0.3em}\\ 
						\leq  \big\vert  \big[ \nabla_y f- \coDifS\coDifO^\top (\coDifO\coDifO^\top)^{-1} \nabla_y h\big]^\top(t, M(t))\big \vert \trace\big( \qcovm(t) \big). 
				\end{align}                    
				In   addition,  the  ellipticity  of  $(\coDifO\coDifO^*)^{-1}$     implies that 
				\begin{equation}
					\trace\left( \qcovm(t) \left( \nabla_y h^\top (\coDifO\coDifO^\top)^{-1} \nabla_y h\right) (t, M(t)) \qcovm(t) \right) \geq
					\trace\left( C \qcovm^{2} (t) \right) \geq \dfrac{C}{n}  \trace\left(  \qcovm(t)  \right)^{2},
				\end{equation}
				with  the  last inequality  coming   from the  property \eqref{trace2}.  
				Thus,  we  can   write   
				\begin{equation}
					-\trace\left( \qcovm(t) \left( \nabla_y h^\top (\coDifO\coDifO^\top)^{-1} \nabla_y h\right) (t, M(t)) \qcovm(t) \right) \leq  \dfrac{C}{n}  \trace\left(  \qcovm(t)  \right)^{2}.
				\end{equation}
				Finally, $\af(t, M(t))$  being  bounded,  we can   find three positive  constants $C_1,~C_2, ~C_3$   such  that 
				\begin{equation}
					\trace(  \dot{\qcovm}(t)) \leq -C_1 \trace\left(  \qcovm(t)  \right)^{2} + C_2 \trace\left(  \qcovm(t)  \right)+ C_3.
				\end{equation} 
				Denoting  $\trace\left(  \qcovm(t)  \right)$ by $p(t)$, the  last inequality  reads
				\begin{equation}\label{eqdot1}
					\dot{p}(t) \leq -C_1 p^2(t) + C_2 p(t) +  C_3.
				\end{equation}
				Then,  since  $\qcovm(t)$  is    bounded,  there exists $C_4$  such  that $C_2 p(t) +  C_3 \leq C_4$.
				Therefore, 
				\begin{equation} \label{eqdot2}
					\dot{p}(t) \leq -C_1 p^2(t) + C_4.
				\end{equation} 
				Then,  the  solution  of \eqref{eqdot1} with equality sign  will   be  less than   the  solution  of  \eqref{eqdot2}  with   equality sign   and  both starting at the  same   initial  condition $p(0)$.
				
				Similarly, if    $\overline{p}(t)$ denotes   the  trace  of   $\qcovm^{-1}(t)$,      writing   the equation  of  the  dynamic   of $\qcovm^{-1}$,  and relying on the assumptions that  $\nabla_y h$
				is   bounded, $\af$  and $(\coDifO\coDifO^\top)^{-1}$  are   elliptic,   we  can  find two  constants  $C'_1 $  and  $C'_2$  such  that  the trace  of  $\qcovm^{-1}(t)$ is  less  than  the   solution  of 			
				\begin{equation} \label{eqdot3}
					\dot{\overline{p}}(t) = -C'_1 \overline{p}^2(t) + C'_2.
				\end{equation} 			
				To  obtain  equation \eqref{eqdot3}, we first derive the dynamics of $\qcovm^{-1}$  by   taking   advantage  of  the fact that 
				$~\qcovm(t)\qcovm^{-1}(t)= I $. This implies that          
				$\dot{\qcovm}(t)\qcovm^{-1}(t)+ \qcovm(t) \dot{\qcovm}^{-1}(t)= 0$   and  thus  $\dot {\qcovm}^{-1}(t)=- \left(\qcovm^{-1} \dot{\qcovm}\qcovm^{-1}\right)(t)$.
				
				Hence, substituting $\dot{\qcovm}(t)$  it   yields  			
				\begin{align*} 
					\dot{\qcovm}^{-1}(t)=&-\qcovm^{-1}(t)\af(t,M(t))\qcovm^{-1}(t)					
					-\qcovm^{-1}(t)\Big[ \nabla_y f - \coDifS \coDifO^\top(\coDifO\coDifO^\top)^{-1} \nabla_y h\Big](t,M(t)) \\&
					-\Big[ \nabla_y f - \coDifS \coDifO^\top(\coDifO\coDifO^\top)^{-1}  \nabla_y h\Big]^{\top}(t,M(t))\qcovm^{-1}(t)
					+\Big[ \nabla_y h^\top(\coDifO\coDifO^\top)^{-1}  \nabla_yh\Big](t,M(t))  .
				\end{align*}
				Applying  the  trace  and  using  the  same  arguments  as  above, \eqref{eqdot3} holds.
				
				Now,  we  have $ \trace(\qcovm(t)) \leq p(t)$  and $ \trace(\qcovm^{-1}(t)) \leq \overline{p}(t)$.  This  implies  that \begin{equation}
					\overline{p}^{-1}(t) \leq \trace(\qcovm(t)) \leq p(t),
				\end{equation}
				and  then 
				\begin{equation}
					\overline{p}^{-1}(t)I \leqm \qcovm(t) \leqm p(t) I.
				\end{equation} 
				So    we  only  have  to  prove  that  $ \overline{p}^{-1}(t) \geq C p(t)$  or  equivalently  that $p(t)\overline{p}(t)$   is  bounded. We  have 			
				\begin{align*}
					\dfrac{d}{dt}\left( p(t)\overline{p}(t )\right)~=& ~\dot{p}(t)\overline{p}(t )+ p(t)\dot{\overline{p}}(t )\\
					~=& ~(-C_1 p^2(t) + C_2) \overline{p}(t)+ p(t) (-C'_1 \overline{p}^2(t) + C'_2 )\\
					~=& ~ (-C_1 p(t) -C'_1 \overline{p}(t)) p(t)\overline{p}(t) + C_2 \overline{p}(t)+ C'_2 p(t)\\
					~\leq &~-C_3( p(t) + \overline{p}(t)) p(t)\overline{p}(t) + C_4( \overline{p}(t)+ p(t))\\
					~\leq &~C_3( p(t) + \overline{p}(t)) \Big(\dfrac{C_4}{C_3}-p(t)\overline{p}(t)\Big)\\
					~\leq &~C_3( p(t) + \overline{p}(t)) (C_5-p(t)\overline{p}(t)),
				\end{align*}
				for  some positive constants $C_3, C_4, C_5$.  Hence,  the derivative of $p(t)\overline{p}(t)$  is  negative as  soon  as $p(t)\overline{p}(t)$  is  greater  than  $C_5$. Moreover,   from  the  assumption  (A\ref{cond4}) on the eigenvalues  of $\qcovm(0)$,  $p(0)\overline{p}(0)$    is  bounded, thus $p(t)\overline{p}(t)$   is  bounded as well. 
				Therefore, the  conclusion  of  Lemma \ref{Lemma_2.2.2}  holds.				
			\end{proof}

			\section{Impact of the initial   error on the EKF estimation error  }
			To implement the EKF procedure, we must assume that the initial value of the hidden state $Y(0)$ is normally distributed with mean $M(0)$ and covariance matrix $\covm(0)$. Therefore, initial guesses for $M(0)$ and $\covm(0)$ are required. These initial guesses generate an initial error, which could impact the convergence of the extended Kalman filter to the true signal. It is necessary to analyze how the initial error evolves over time. The theorem below demonstrates that under certain technical conditions, the initial error diminishes exponentially fast.  
			\begin{theorem}  \label{theo2}
				Let $(M(t),\covm(t))_{t \geq 0}$  be  an  extended  Kalman  filter for the filtering problem \eqref{FiltPb},  with the  gain $G$ defined as  \begin{equation}
					G(t) = \left[   g (t) \ell^\top (t, M(t)) + \qcovm(t) \nabla_y h^\top (t, M(t)) \right] (\ell\ell^\top)^{-1}(t, M(t)),
				\end{equation}
				 where $\qcovm=\frac{1}{\varepsilon} \covm$.
				If   the  following  assumptions are fulfilled,
				
					\begin{enumerate}[wide, labelwidth=!, labelindent=0pt]			\renewcommand{\labelenumi}{(B\theenumi)} 	
						\item\label{condB1} $\sigma (t, Y(t))$  and $g(t, Y(t))$  are  bounded in $L^{\infty -},$
						\item \label{condB2}the functions $f$  and $h$ are $\mathcal{C}^1$  with  bounded derivatives,
						\item 
						\label{condB3}\hspace*{-2mm}\begin{tabular}[t]{l}						
							the  observable  process  $G(t)$   is  such  that  for  any $\mathbb{F}^Z$-adapted process $\xi$,  the  process \\[0.5ex]
							\hspace*{5em}$
							A(t) =\nabla_yf (t, \xi(t))- G(t) \nabla_y h  (t, \xi(t)) 
							$\\[0.5ex]
							is  exponentially stable. More precisely,  it  satisfies  \eqref{expstab} for  some  $c >0$.
						\end{tabular}
					\end{enumerate}	
					Then there exists a constant $C_q>0$ such that for all $c_0 < c$ it holds					
					\begin{equation}
						\vert \sig(t)-M(t) \vert_q \leq C_q  \sqrt{\varepsilon} + C_q  \vert \sig(0)-M(0) \vert_q e^{ - c_0 t}.
					\end{equation}				
				\end{theorem}			
				The conclusion of Theorem \ref{theo2}  shows that for large $t$, the first  term $C_q  \sqrt{\varepsilon}$ is the dominating part of the upper bound  on the right-hand side.  This implies that for $t$ large enough   the filtering error is uniformly bounded w.r.t. time.
				The proof of Theorem \ref{theo2} involves the use of the following lemma, which is proven in \cite{picard1991efficiency}.
				
				\begin{lemma} \label{lem2}
					Consider  a family $A(t)$ of    locally  bounded processes with  values in the class of square  matrices of  order $n$,
					
						\begin{enumerate}
							\item Suppose  that there  exist  a uniformly  bounded   and  elliptic  family $\K$,  and  a  constant number $k> 0$  such  that  $A$  is  $(\K, k)$-stable,  then $A$  is exponentially  stable and the  estimate \eqref{expstab} is  satisfied  with  $c=k/2$.
							\item Conversely,  if $A$ is  uniformly  bounded   and  exponentially   stable (so  that   estimate \eqref{expstab}  holds  for some $c> 0$)  then  for  any $k< 2c $,  there exists a  family of  uniformly bounded and  elliptic  processes $\K$ that   are  adapted  to  the  filtration  generated by $A$   and  are  such  that  $A$  is  $(\K, k)$-stable. More  generally, if  $q_0$ is a  family of  uniformly  positive numbers,  there  exists a family  of  uniformly  elliptic  processes $\K$   such  that $\K(0)= q_0 I,$   $A$  is  $(\K, k)$-stable  and for $t>0$ 
							\begin{equation}
								\trace(\K(t)) \leq C \left(1 + q_0 e^{ (-(2c -k)t) } \right).
							\end{equation}
						\end{enumerate}
					\end{lemma}
					
					\begin{proof} (of Theorem \ref{theo2})
					Let $(M(t),\covm(t))_{t \geq 0}$   be  the extended   Kalman  filter of  the  model with gain $G(t)$   defined  in  the  theorem  and initial guess $M(0)$.			
						Since  $f$  and $h$ are $C^1$ with bounded derivatives  by  Assumption (B\ref{condB2}),  there exists a  $\mathbb{F}$-adapted process $\xi$ such  that 
						\begin{align}
							f (t,\sig(t))-f (t,M(t))& = \nabla_yf (t,\xi(t)) \left(\sig(t)-M(t)\right), \quad \text{and}\\						
							h (t,\sig(t))-h (t,M(t))&= \nabla_yh (t,\xi(t)) \left(\sig(t)-M(t)\right).
						\end{align}						
						Thus,  defining $\x(t) := \sig(t)-M(t) $, it follows that 						
						\begin{align*}
							f (t,\sig(t))-f(t,M(t))& - G(t) \left( h(t,\sig (t))-h (t,M(t))\right)\\
							&= \left( \nabla_yf (t, \xi(t))- G(t)  \nabla_yh (t, \xi(t)) \right) \x(t)
							= A(t)\x(t),
						\end{align*} 
						with  $A(t) :=  \nabla_yf (t, \xi(t))- G(t)  \nabla_yh (t, \xi(t)).$\\						
						Defining  $\widetilde{W}(t) := \left( W^{(1)}(t), W^{(2)}(t) \right)^\top$,  
						we have 
						\begin{align*}
						d\x(t)=&d\left(\sig(t)-M(t) \right)\\
							=& \left[ f (t,\sig(t))-f (t,M(t))- G(t)\left( h(t,\sig(t))-h (t,M(t))\right) 
							\right]dt\\
							&+\sqrt{\eps} \left[ \sigma (t,\sig(t)) d W^{(1)}(t) + \left(\coDifS (t,\sig(t)) - G(t) \coDifO (t,\sig(t)) \right)  dW^{(2)}(t) \right] \\
							=&\left( \nabla_yf (t, \xi_t)- G(t) \nabla_yh (t, \xi(t)) \right) \x(t)dt \\
							&+\sqrt{\eps}  \left( \sigma (t,\sig(t)) ,~ \coDifS (t,\sig(t)) - G(t) \coDifO (t,\sig(t)) \right)  d\widetilde{W}(t) \\
							=&A(t)\x(t) dt + \sqrt{\eps}D (t) d\widetilde{W}(t),
						\end{align*}						
						with $D(t):=\left( \sigma (t,\sig(t)) ,~ \coDifS (t,\sig(t)) - G(t) \coDifO (t,\sig(t)) \right).$
											
						Since $A$  is  exponentially  stable  and  bounded by  Assumption (B\ref{condB3}), Lemma \ref{lem2}
						asserts  that  for any  $k< 2c $,  there  exists a family of  processes $ \K$  that  are uniformly   bounded, elliptic  and  adapted  such  that  $A$ is  $(\K, k)$    stable,  and  for  some  family  of   uniformly  positive  numbers $q_0= \vert \sig(0)-M(0)\vert_q^2$, $ \K(t)$  satisfy 
						\begin{equation}
							\trace(\K(t)) \leq C \big(1 + q_0 e^{-(2c -k)t } \big).
						\end{equation}
						We deduce also  that  $A$ is  $(\eps \K, k)$  stable.
						We are  then  in  the  setting of Lemma \ref{Lemma_1.3.2} with $F(t)= 0$. Thus  for  $\a=0$			
						and because $\sigma, \coDifS$  and $\coDifO$  are  bounded  by  Assumption (B\ref{condB1}),  for  any  $q \geq 1$  and  $\eps$   small  enough,   if $\x^\top(0) (\eps\K(0))^{-1} \x(0)$  is   bounded  in   $L^q$,  then  the  process   $\x^\top(t) (\eps \K(t))^{-1} \x(t)$  is  bounded  in  $L^q$.  										
						In  addition,  						
						$\eps \K(t) \leqm \trace(\eps  \K(t)) I$  \text{implies}   $(\trace(\eps  \K(t)))^{-1} I \leqm (\eps \K(t) )^{-1}$,
						thus 
						$$\x^\top(t) (\trace(\eps \K(t)))^{-1} \x(t)  \leq \x^\top(t) (\eps \K(t))^{-1} \x(t).$$ 
						It follows that  
						$x^\top(t)(\trace(\eps \K(t)))^{-1} \x(t)$  \text{  is  bounded  in }  $L^q$.
						Finally,   given  the  fact  that\\ $\x^\top(t)(\trace(\eps \K(t)))^{-1} \x(t) = \vert (\trace(\eps \K(t)))^{-1/2} \x(t) \vert^2 $, we deduce  that $(\trace(\eps\K(t)))^{-1/2}\x(t)$  is  bounded  in $L^q$.						
						Therefore, there  exists $C_q$  such  that 
						\begin{equation}
							\vert (\trace(\eps \K(t)))^{-1/2} \x(t) \vert_q \leq C_q,
						\end{equation}
						This  implies, using the inequality $(1+x)^{1/2} \leq 1+ x^{1/2}$ for  $x \geq 0$,  that 
						\begin{align*}
							\vert \x(t)\vert_q ~~\leq  \sqrt{\eps}C_q (\trace(\K(t)))^{1/2}  
							&\leq  \sqrt{\eps} C_q  \left(1 + q_0 e^{ -(2c -k)t}  \right)^{1/2}\\
							&\leq  C_q \sqrt{\eps} +   C_q \sqrt{\eps} \vert \sig(0)-M(0) \vert_q e^{ -(c -k/2)t}\\
							&\leq  C_q \sqrt{\eps} +   C_q  \vert \sig(0)-M(0) \vert_q e^{ -c_0 t},
						\end{align*}					
						with $c_0= c -k/2$  and  $\eps$  small enough.
						
					\end{proof}

					\subsection*{Acknowledgments}			The authors thank  the collaborators within the DFG research project ``CESMO - Contain Epidemics with Stochastic Mixed-Integer Optimal Control'', in particular Gerd Wachsmuth,   Armin Fügenschuh, Markus Friedemann, Jesse Beisegel (BTU Cottbus--Senftenberg),	for insightful discussions and valuable suggestions that improved this paper.

						\smallskip\noindent
						\textbf{Funding~} 	
						The  authors gratefully acknowledge the  support by the Deutsche Forschungsgemeinschaft (DFG), award number 458468407,  and by the  German Academic Exchange Service (DAAD), award number 57417894.

					\appendix
					\section*{Appendix}
					\addcontentsline{toc}{section}{Appendix}
					\renewcommand{\thesection}{\Alph{section}}
					
					\section{The Moore-Penrose pseudoinverse}\label{pseudoinvsec}
					
					Let $n,m$ be two positive integers, and $\mathcal{M}_{n,m}(\mathbb{R})$  the class of $n\times m $  matrices with real entries.
					Every   $A \in \mathcal{M}_{n,m}(\mathbb{R})$
					has a pseudoinverse and, moreover,
					the pseudoinverse, denoted  $A^+ \in \mathcal{M}_{n,m}(\mathbb{R})$, is unique. A purely algebraic characterization of
					$A^+$ is given in the next theorem proved by Penrose in 1956.

					\begin{theorem}	
						Let $A \in \mathcal{M}_{n,m}(\mathbb{R})$. Then $B = A^+$ if and only if	
						\begin{enumerate}		
							\item  $ABA = A$,\\[-4ex]
							\item $BAB = B$,\\[-4ex]
							\item $(AB)^\top = AB$,\\[-4ex]
							\item $(BA)^\top = BA$.		
						\end{enumerate}
						Furthermore, $A^+$ always exists and is unique
					\end{theorem}
					
					\section{General  results  on symmetric  positive definite (semi-definite) matrices }
					We refer  to  the  book  \cite{horn2012matrix}	for  further  details  and proof of  results in  this  section. Let $n,m$ be two positive integers, and $\mathcal{M}_n(\mathbb{R})$  the class of $n\times n $  matrices with real elements:
					
					\begin{defination}
						A  symmetric  matrix  $A \in \mathcal{M}_n(\mathbb{R})$ is 
							\begin{enumerate}
								\item positive definite if  $x^\top Ax>0$  for  all  nonzero $ x \in \mathbb{R}^n$;
								\item positive semi-definite if  $x^\top Ax \geq 0$  for  all  nonzero $ x \in \mathbb{R}^n$.
							\end{enumerate}
						\end{defination}
						
						\begin{proposition}
							Let $A_1, \ldots , A_k \in  \mathcal{M}_n(\mathbb{R})$ be positive semi-definite and let $\alpha_1, \ldots , \alpha_k$
							be nonnegative real numbers. Then  $\sum_{i=1}^k \alpha_i A_i$ is positive semi-definite. If there is a
							$j \in  \lbrace 1, \ldots , k \rbrace$ such that $\alpha_j > 0 $ and $A_ j$ is positive definite, then $\sum_{i=1}^k \alpha_i A_i$ is positive
							definite.
						\end{proposition}	
						
						\begin{proposition}
							Let $A \in  \mathcal{M}_n(\mathbb{R})$ be positive semi-definite (respectively, positive definite).
							Then $\trace A$, $\det A$, and the principal minors of $A$ are all nonnegative (respectively,
							positive). Moreover, $\trace A = 0$ if and only if $A = 0$.
						\end{proposition}	
						
						\begin{proposition}
							Let $A \in  \mathcal{M}_n(\mathbb{R})$ be positive semi-definite and let $ x \in  \mathbb{R}^n$. Then $x^\top Ax =0$ if and only if $Ax = 0$.
							
						\end{proposition}
						
						\begin{proposition}
							A positive semi-definite matrix is positive definite if and only if it is
							nonsingular.
						\end{proposition}
						
						\begin{theorem}
							Let $A \in \mathcal{M}_n(\mathbb{R})$ be symmetric. Then $x^\top Ax > 0$ (respectively,
							$x^\top Ax \geq  0$) for all nonzero $x \in \mathbb{R}^n$ if and only if every eigenvalue of $A$ is positive
							(respectively, nonnegative).
						\end{theorem}
						
						\begin{theorem}
							A symmetric matrix is positive semi-definite if and only if all of its
							eigenvalues are nonnegative. It is positive definite if and only if all of its eigenvalues
							are positive.
						\end{theorem}
						
						\begin{corollary}
							A nonsingular symmetric matrix $A \in \mathcal{M}_n(\mathbb{R})$	 
							is positive definite if and only if $ A ^{-1}$ is positive definite.
						\end{corollary}

						\begin{corollary}
							If $A \in  \mathcal{M}_n(\mathbb{R})$ is positive semi-definite, then so is each $A^k, ~k = 1, 2,\ldots $.	
							
						\end{corollary}

						\begin{theorem}
							Let $A \in  \mathcal{M}_n(\mathbb{R})$ be symmetric.
								\begin{enumerate}
									\item  A is positive semi-definite if and only if there is a $B \in  \mathcal{M}_{m,n}(\mathbb{R}) $ such that $A = B^\top B$.
									\item If $A = B^\top B$  with  $B \in  \mathcal{M}_{m,n} $, and if $x \in  \mathbb{R}^n$, then $Ax = 0$  if and only if $Bx = 0$,
									so $\nullspace A = \nullspace B$ and $\rank A = \rank B$.
									\item  If $A = B^\top B$ with $B \in  \mathcal{M}_{m,n}(\mathbb{R}) $, then $A$ is positive definite if and only if  $B$ has full
									column rank.
									
							\end{enumerate}\end{theorem}
							
							\section{The Loewner partial order}	
							
							\begin{proposition}
								
								If $A \in  \mathcal{M}_{n}(\mathbb{R})$ is symmetric with the smallest and largest eigenvalues
								$\lambda_{min}(A)$  and $\lambda_{max}(A)$, respectively, then  $\lambda_{max}(A)I \geqm  A \geqm \lambda_{min}(A) I$
							\end{proposition}	
							
							\begin{proposition}
								Let $A \in  \mathcal{M}_{n}(\mathbb{R})$ be symmetric, 
									\begin{enumerate}
										\item $I \geqm  A$ if and only if $\lambda_{max} (A) \leq 
										1$; 
										
										\item  $ I \gm  A$, if and only if $\lambda_{max}(A) < 1$.
										
									\end{enumerate}
								\end{proposition}

\begin{defination}
	Let $X \in  \mathcal{M}_{n,m}(\mathbb{R})$. Then $\sigma_1(X) = \lambda_{ max}^{1/2}(X X ^\top) = \lambda_{max}^{1/2}(X^\top X) = \sigma_1(X^\top)$ is the
			largest singular value (the spectral norm) of $X$. We say that $X$ is a contraction if
		$\sigma_ 1(X) \leq  1$; it is a strict contraction if $\sigma_ 1(X) < 1$.
\end{defination}
	\begin{theorem}
								
Let $A,~ B \in  \mathcal{M}_{n}(\mathbb{R})$  be symmetric and let $S \in  M_{n,m}$. Then

\begin{enumerate}
\item  if $A \geqm B$, then $S^\top AS \geqm  S^\top BS$;
\item  if $\rank S = m$, then $A \gm B$ implies $S^\top AS \gm  S^\top BS$
\item if $m = n$ and $S \in M_n$ is nonsingular, then $A \gm B $ if and only if $S^\top AS \gm  S^\top BS$;
									$A \geqm B $  if and only if $S^\top AS \geqm  S^\top BS$;
\item  $I_m \gm S^\top S$ (respectively, $I_n\gm SS^\top$) if and only if $S$ is a strict contraction; $I_m \geqm S^\top S$ (respectively, $I_n\geqm SS^\top$) if and only if $S$ is a  contraction;
									
										\end{enumerate}
									\end{theorem}
									
\begin{theorem}
Let $A, B \in \mathcal{M}_{n}(\mathbb{R})$ be symmetric. Let $\lambda_1(A) \leq \ldots \leq \lambda_n(A)$ and $\lambda_1(B) \leq \ldots \leq \lambda_n(B)$ be the ordered eigenvalues of A and B, respectively.
										
\begin{enumerate}
\item If $A \gm  0$ and $B \gm 0$, then $A \geqm B$ if and only if $ B^{-1} \geqm A^{-1}$.
\item If $A \gm  0$, $B \geqm 0$, and $A \geqm  B$ , then $A^{1/2} \geqm  B^{1/2}$.
	\item If $A \geqm B$, then $ \lambda_i(A)\geq \lambda_i(B)$ for each $i = 1, \ldots , n$.
	\item If $A \geqm  B$, then $\trace A \geq  \trace B$ with equality if and only if $A = B$.
	\item If $A \geqm B \geqm 0$, then $\det A \geq  \det B \geq 0$.
		\end{enumerate}
	\end{theorem}
										
	\begin{theorem}
Let $A \in \mathcal{M}_{n}(\mathbb{R})$ be positive semi-definite and let $B \in  \mathcal{M}_{n}(\mathbb{R})$ be symmetric.
The following four statements are equivalent:
											
\begin{enumerate}
\item \label{prop1} $x^\top Ax \geq  \vert x^\top Bx \vert$  for all $x \in \mathbb{C}^n$.
\item $x^\top Ax + y^\top Ay \geq  2\vert x^\top By\vert$ for all $x, y \in  \mathbb{C}^n$.
\item $H =\left[ \begin{array}{cc} 
A&B\\
	B&A
\end{array} \right]$  is positive semi-definite.
	\item There is a symmetric contraction $X \in \mathcal{M}_{n}(\mathbb{R})$ such that $B = A^{1/2} X A^{1/2}$.

    If A is positive definite, then the following statement is equivalent to \ref{prop1}:
	\item $\rho (A^{-1} B) \leq  1$.
	\end{enumerate} \end{theorem}
											
\section{Some  properties  of  the  trace}
											
We  refer  to  \cite{gabih2020asymptotic}  for  the  proof  of  the  following  lemma.
											
\begin{lemma}
(Properties of symmetric and positive semi-definite matrices)
Let $A, B \in \mathcal{M}_{d}(\mathbb{R})$, $d \in \mathbb{N}$, symmetric and positive semi-definite matrices.
	Then it holds
												
\begin{enumerate}
											
\item 														
\begin{equation}
\lambda_{min}(A) \trace(B) \leq \trace(AB)  \leq  \lambda_{max}(A) \trace(B)
\end{equation}
where $\lambda_{min}(A)$  and $\lambda_{max}(A)$ denote the smallest and largest eigenvalue
	of A, respectively.
												
	\item  \begin{equation} \label{trace1}
	\dfrac{\trace(B)}{ \trace\left( A^{-1} \right)} \leq   \trace(AB) \leq \trace(A)  \trace (B) 
						\end{equation}														
where for the first inequality A is assumed to be positive definite.
														
\item \begin{equation}\label{trace2}
	\trace^2(A)  \geq \trace\left( A^2 \right) \geq  \dfrac{1}{d} \trace^2 (A)
\end{equation}
														
\item  \begin{equation}				\vert A\vert_F = \sqrt{\trace\left( A^2 \right)}  \leq \trace(A)
						\end{equation}
			where $\vert A\vert_F$  denotes the Frobenius norm of A.
					\end{enumerate}
													
							\end{lemma}
									
					\begin{theorem} 			Let $A, B \in \mathcal{M}_{d}(\mathbb{R})$, $d \in \mathbb{N}$, symmetric and positive semi-definite matrices, then 
					\begin{equation}
					\trace(AB) \leq   \vert \trace(AB)\vert \leq \vert A \vert_2  \trace (B) 
					\end{equation}
					where $\vert A \vert_2 $   denotes the  spectral  norm  or  largest  singular  value  of  $A$.
									\end{theorem}


\begin{thebibliography}{aa}

\bibitem{bain2008fundamentals}
  {Bain, Alan and Crisan, Dan},  {Fundamentals of stochastic filtering},
  {60},
  {2008},
  {Springer Science \& Business Media}.

\bibitem{BrittonPardoux2019}		
	{Tom Britton and Etienne Pardoux}, {Stochastic Epidemic Models with Inference}, 
	    2019,
	 {Springer Internat. Publ.}.

\bibitem{katzur1984asymptotic1} {Katzur, R. and Bobrovsky, BZ and Schuss, Z.},
	{Asymptotic analysis of the optimal filtering problem for one-dimensional diffusions measured in a low noise channel, part {I}},	
	{SIAM Journal on Applied Mathematics},
	{44},
	{3},
	{591--604},
	{1984},
	{SIAM}.

\bibitem{katzur1984asymptotic2} 
	{Katzur, R. and Bobrovsky, BZ and Schuss, Z.},
    {Asymptotic analysis of the optimal filtering problem for one-dimensional diffusions measured in a low noise channel, Part {II}},
	{SIAM Journal on Applied Mathematics},
	{44},
	{6},
	{1176--1191},
	{1984},
	{SIAM}.


\bibitem{joannides1995nonlinear} 
	{Joannides, Marc and LeGland, Francois},{Nonlinear filtering with perfect discrete time observations},
	{Proceedings of 1995 34th IEEE Conference on Decision and Control},
	{4},
	{4012--4017},
{1995},
{IEEE}.
\bibitem{joannides1997nonlinear}  
	{Joannides, Marc and LeGland, Francois}, {Nonlinear filtering with continuous time perfect observations and noninformative quadratic variation},	
	{Proceedings of the 36th IEEE Conference on Decision and Control},
	{2},
	{1645--1650},
	{1997},
	{IEEE}.
\bibitem{takeuchi1981least}
	{Takeuchi, Yoshiki and Akashi, Hajime},{Least-squares state estimation of systems with state-dependent observation noise},
	{IFAC Proceedings Volumes},
	{14},
	{2},
	{557--562},
	{1981},
	{Elsevier}.
\bibitem{takeuchi1985least}
	{Takeuchi, Yoshiki and Akashi, Hajime},{Least-squares state estimation of systems with state-dependent observation noise},
	{Automatica},
	{21},
	{3},
	{303--313},
	{1985},
	{Elsevier}.

\bibitem{mclane1969optimal}
	{McLane, PJ},{Optimal linear filtering for linear systems with state-dependent noise},
	{International Journal of Control},
	{10},
	{1},
	{41--51},
	{1969},
{Taylor \& Francis}.
\bibitem{1099828}
  {McLane, P.}, {Optimal stochastic control of linear systems with state- and control-dependent disturbances},
  {IEEE Transactions on Automatic Control}, 
  {1971},
  {16},
 {6},
 {793-798}.


\bibitem{kutoyants2025extended}
 {Kutoyants, Yury A}, {Extended adaptive {K}alman filter with low noise observations},
  {2025},
 {10},
 {4},
{443-470},
 {10.3934/puqr.2025023},
{Probability, Uncertainty and Quantitative Risk}.


\bibitem{gelb1974applied}
 {Gelb, Arthur},  {Applied optimal estimation},
  {1974},
 {MIT Press}.

\bibitem{LiptserShiryaevVolII2001}
	{Liptser, Robert S. and  Albert N. Shiryaev},{Statistics of Random Processes II: Applications},
	{00041918},
	{Applications of Mathematics Stochastic Modelling and Applied Probability Series},
	{2001},
		{Springer}.

     
\bibitem{horn2012matrix}
{Horn, Roger A and Johnson, Charles R},  {Matrix analysis},
  {2012},
 {Cambridge University Press}.

\bibitem{gabih2020asymptotic}{Gabih, Abdelali and Kondakji, Hakam and Wunderlich, Ralf}, {Asymptotic filter behavior for high-frequency expert opinions in a market with {G}aussian drift},
  {Stochastic Models},
  {36},
  {4},
  {519--547},
  {2020},
 {Taylor \& Francis}.
\bibitem{njiasse2025stoch}
   {Mbouandi Njiasse, Ibrahim and  Ouabo Kamkumo, Florent  and  Wunderlich, Ralf},  {Stochastic Optimal Control of an Epidemic Under Partial Information}, 
      {arxiv:2503.06804},
    {2025}.



\bibitem{mbouandi2025phd}
 {Mbouandi Njiasse, Ibrahim},  {Stochastic models and optimal control of epidemics under partial information},
  {2025},
{https://doi.org/10.26127/BTUOpen-7127},
{BTU Cottbus-Senftenberg}.

\bibitem{ouabo2025phd}
 {Ouabo Kamkumo, Florent},  {Stochastic Epidemic Models with Partial Information and Dark Figure Estimation },
 {2025},
 {BTU Cottbus-Senftenberg},
   {https://doi.org/10.26127/BTUOpen-7197}.

\bibitem{kamkumo2025stochastic}
 {Ouabo Kamkumo, Florent  and Mbouandi Njiasse, Ibrahim  and Wunderlich, Ralf},  {Stochastic Epidemic Models with Partial Information},
 { arXiv:2503.07251},
  {2025}.

\bibitem{kamkumo2025estimating}
{Ouabo Kamkumo, Florent  and Mbouandi Njiasse, Ibrahim  and Wunderlich, Ralf}, {Estimating Unobservable States in Stochastic Epidemic Models with Partial Information},
  {arXiv:2506.00906},
 {2025}.  

\bibitem{pardoux1991filtrage}
Pardoux, {\'E}tienne, 
	{Filtrage non lin{\'e}aire et {\'e}quations aux d{\'e}riv{\'e}es partielles stochastiques associ{\'e}es},	
	{Ecole d'Et{\'e} de Probabilit{\'e}s de Saint-Flour XIX-1989},
	{68--163},
	{1991},
	{Springer}.

\bibitem{picard1991efficiency}
  {Picard, Jean},  {Efficiency of the extended {K}alman filter for nonlinear systems with small noise},
  {SIAM Journal on Applied Mathematics},
  {51},
 {3},
  {843--885},
  {1991},
  {SIAM}.
\bibitem{picard1993estimation}
{Picard, Jean}, {Estimation of the quadratic variation of nearly observed semi-martingales with application to filtering},
	{SIAM Journal on Control and Optimization},
	{31},
	{2},
	{494--517},
	{1993},
	{SIAM}.





\end{thebibliography}

\end{document}